\documentclass[preprint,11pt]{elsarticle}
\usepackage{makeidx}
\usepackage{amssymb}
\usepackage{amsfonts}
\usepackage{amsmath}

\newtheorem{theorem}{Theorem}
\newtheorem{corollary}[theorem]{Corollary}
\newtheorem{definition}[theorem]{Definition}

\newtheorem{lemma}[theorem]{Lemma}
\newtheorem{proposition}[theorem]{Proposition}
\newtheorem{remark}[theorem]{Remark}
\newenvironment{proof}[1][Proof]{\noindent\textbf{#1.} }{\ \rule{0.5em}{0.5em}}
\input{tcilatex}
\begin{document}

\begin{frontmatter}



\title{On the Rate of Convergence of Weak Euler Approximation for Non-degenerate SDEs}


\author{R. Mikulevi\v{c}ius and C. Zhang}

\address{University of Southern California, Los Angeles, USA}

\begin{abstract}
The paper estimates the rate of convergence of the weak Euler approximation for solutions to SDEs
driven by point and martingale measures, with H\"{o}lder continuous coefficients. The equation considered
has a non-degenerate main part whose jump intensity measure is absolutely continuous with respect to
the L\'{e}vy measure of a spherically-symmetric stable process. It includes the nondegenerate diffusions
and SDEs driven by L\'{e}vy processes.

\end{abstract}


\end{frontmatter}




\section{\textrm{Introduction}}


In this paper we consider the weak Euler approximation for solutions to SDEs
driven by point and martingale measures. It is a continuation of \cite{MiZ10}
where some Markov It\^{o} processes were approximated. Let $\alpha \in (0,2]$%
\ be fixed. In a complete probability space $(\Omega ,\mathcal{F},\mathbf{P})
$ with a filtration $\mathbb{F}=\{\mathcal{F}_{t}\}_{t\in \lbrack 0,T]}$ of $%
\sigma $-algebras satisfying the usual conditions, we consider an $\mathbb{F}
$-adapted $d$-dimensional stochastic process $X_{t},t\in \lbrack 0,T],$\
solving for $t\in \lbrack 0,T]$ 
\begin{eqnarray}
X_{t} &=&X_{0}+\int_{0}^{t}a_{\alpha }(X_{s-})ds+\int_{0}^{t}b_{\alpha
}(X_{s-})dW_{s}  \label{one} \\
&&+\int_{0}^{t}\int_{|y|>1}c(X_{s-})h_{\alpha }(X_{s-},\frac{y}{|y|}%
)yp_{0}(ds,dy)  \notag \\
&&+\int_{0}^{t}\int_{|y|\leq 1}c(X_{s-})h_{\alpha }(X_{s-},\frac{y}{|y|}%
)yq_{0}(ds,dy)  \notag \\
&&+\int_{0}^{t}\int_{U_{1}^{c}}l_{\alpha
}(X_{s-},v)p(ds,dv)+\int_{0}^{t}\int_{U_{1}}l_{\alpha }(X_{s-},v)q(ds,dv), 
\notag
\end{eqnarray}%
where $W_{t},t\in \lbrack 0,T],$\ is a $d$-dimensional $\mathbb{F}$-adapted\
standard Wiener process, $p_{0}(dt,dy)$ and $p(dt,d\upsilon )$\ are
independent Poisson point measures on $[0,T]\times \mathbf{R}_{0}^{d}$ ($%
\mathbf{R}_{0}^{d}=\mathbf{R}^{d}\backslash \{0\}$) and $[0,T]\times U$
respectively with 
\begin{eqnarray*}
q_{0}(dt,dy) &=&p_{0}(dt,dy)-\frac{dy}{|y|^{d+\alpha }}dt, \\
q(dt,d\upsilon ) &=&p(dt,d\upsilon )-\pi (d\upsilon )dt,
\end{eqnarray*}%
being the corresponding martingale measures and 
\begin{eqnarray*}
a_{\alpha }(x) &=&\mathbf{1}_{\{\alpha \in (0,1)\}}\Big(\int_{|y|\leq
1}c(x)h_{\alpha }(x,\frac{y}{|y|})y\frac{dy}{|y|^{d+\alpha }}%
+\int_{U_{1}}l_{\alpha }(x,v)\pi (dv)\Big) \\
&&+\mathbf{1}_{\{\alpha =1\}}\Big(a(x)+\int_{U_{1}}l_{\alpha }(x,v)\pi (dv)%
\Big) \\
&&+\mathbf{1}_{\{\alpha \in (1,2]\}}\Big(a(x)-\int_{|y|>1}c(x)h_{\alpha }(x,%
\frac{y}{|y|})y\frac{dy}{|y|^{d+\alpha }}\Big), \\
b_{\alpha }(x) &=&\mathbf{1}_{\{\alpha =2\}}b(x),
\end{eqnarray*}%
The coefficient functions $a=(a^{i})_{1\leq i\leq d},\alpha \in \lbrack
1,2],c=(c^{ij})_{1\leq i,j\leq d},h_{\alpha },\alpha \in (0,2),$ and $%
b=(b^{ij})_{1\leq i,j\leq d},\alpha =2,$\ are measurable and bounded, $\pi
(d\upsilon )$ is a non-negative $\sigma $-finite measure on a measurable
space $(U,\mathcal{U})$: there is a sequence $U_{n}\in \mathcal{U}$ such
that $U=\bigcup_{n}U_{n}^{c}$ and $\pi (U_{n}^{c})<\infty $ for each $n.$ We
assume that $l_{\alpha },\alpha \in (0,2]$ is measurable and $%
\int_{U_{1}}|l_{\alpha }(x,\upsilon )|^{\alpha }\pi (d,\upsilon )$ is
bounded. A class of strong Markov processes satisfying (\ref{one}) is
constructed, for example, in \cite{MiP923}, \cite{AbK09} (see references
therein as well). In particular, (\ref{one}) covers a large class of SDEs
driven by L\'{e}vy processes (see subsection 2.3 below).

The process defined in (\ref{one}) is used as a mathematical model for
random dynamic phenomena in applications from fields such as finance and
insurance, to capture continuous and discontinuous uncertainty. It naturally
arises in stochastic differential equations driven by L\'{e}vy processes as
well (see subsection 2.3 below). For many\ applications, the practical
computation of functionals of the type $F=\mathbf{E}g(X_{T})$ and $F=\mathbf{%
E}\int_{0}^{T}f(X_{s})ds$ plays an important role. For instance in finance,
derivative prices can be expressed by such functionals. One possibility to
numerically approximate $F$\ is given by the discrete time Monte-Carlo
simulation of the It\^{o} process $X$. The simplest discrete time
approximation of $X$\ that can be used for such Monte-Carlo methods is the
weak Euler approximation.

Let the time discretization $\{\tau _{i},i=0,\ldots ,n_{T}\}$ of the
interval $[0,T]$ with maximum step size $\delta \in (0,1)$ be a partition of 
$[0,T]$ such that $0=\tau _{0}<\tau _{1}<\dots <\tau _{n_{T}}=T$ and $%
\max_{i}(\tau _{i}-\tau _{i-1})\leq \delta .$ \ The Euler approximation of $%
X $ is an $\mathbb{F}$-adapted stochastic process $Y=\{Y_{t}\}_{t\in \lbrack
0,T]}$ defined for $t\in \lbrack 0,T]$ by the stochastic equation%
\begin{eqnarray}
Y_{t} &=&X_{0}+\int_{0}^{t}\int_{\mathbf{R}_{0}^{d}}c(Y_{\tau
_{i_{s}}})h_{\alpha }(Y_{\tau _{i_{s}}},\frac{y}{|y|})yp_{0}(ds,dy)  \notag
\\
&&+\int_{0}^{t}\int l_{\alpha }(Y_{\tau _{i_{s}}},\upsilon )p(ds,d\upsilon )%
\text{ if }\alpha \in (0,1),  \notag \\
Y_{t} &=&X_{0}+\int_{0}^{t}a(Y_{\tau
_{i_{s}}})ds+\int_{0}^{t}\int_{|y|>1}c(Y_{\tau _{i_{s}}})h_{\alpha }(Y_{\tau
_{i_{s}}},\frac{y}{|y|})yp_{0}(ds,dy)  \label{two} \\
&&+\int_{0}^{t}\int_{|y|\leq 1}c(Y_{\tau _{i_{s}}})h_{\alpha }(Y_{\tau
_{i_{s}}},\frac{y}{|y|})yq_{0}(ds,dy)+\int_{0}^{t}\int l_{\alpha }(Y_{\tau
_{i_{s}}},\upsilon )p(ds,d\upsilon )\text{ if }\alpha =1,  \notag \\
Y_{t} &=&X_{0}+\int_{0}^{t}a(Y_{\tau _{i_{s}}})ds+\int_{0}^{t}\int_{\mathbf{R%
}_{0}^{d}}c(Y_{\tau _{i_{s}}})h_{\alpha }(Y_{\tau _{i_{s}}},\frac{y}{|y|}%
)yq_{0}(ds,dy)  \notag \\
&&+\int_{0}^{t}\int_{U_{1}}l_{\alpha }(Y_{\tau _{i_{s}}},\upsilon
)q(ds,d\upsilon )+\int_{0}^{t}\int_{U_{1}^{c}}l_{\alpha }(Y_{\tau
_{i_{s}}},\upsilon )p(ds,d\upsilon )\text{ if }\alpha \in (1,2),  \notag \\
Y_{t} &=&X_{0}+\int_{0}^{t}a(Y_{\tau _{i_{s}}})ds+\int_{0}^{t}b(Y_{\tau
_{i_{s}}})dW_{s}  \notag \\
&&+\int_{0}^{t}\int_{U_{1}}l_{2}(Y_{\tau _{i_{s}}},\upsilon )q(ds,d\upsilon
)+\int_{0}^{t}\int_{U_{1}^{c}}l_{2}(Y_{\tau _{i_{s}}},\upsilon
)p(ds,d\upsilon )\text{ if }\alpha =2,  \notag
\end{eqnarray}%
where $\tau _{i_{s}}=\tau _{i}$\ if $s\in \lbrack \tau _{i},\tau
_{i+1}),i=0,\ldots ,n_{T}-1.$ Contrary to those in (\ref{one}), the
coefficients in (\ref{two}) are piecewise constants in each time interval of 
$[\tau _{i},\tau _{i+1}).$

The weak Euler approximation $Y$\ is said to converge with order $\kappa >0$%
\ if for each bounded smooth function $g$ with bounded derivatives, there
exists a constant $C$, depending only on $g$, such that 
\begin{equation*}
|\mathbf{E}g(Y_{T})-\mathbf{E}g(X_{T})|\leq C\delta ^{\kappa },
\end{equation*}%
where $\delta >0$\ is the maximum step size of the time discretization.

The cases in which the coefficients are smooth, especially for diffusion
processes ($\alpha =2,\pi =0),$ have been considered by many authors.
Milstein (see \cite{Mil79, Mil86}) was one of the first to study the order
of weak convergence for diffusion processes (\ref{exf1}) with $\alpha =2$
and derived $\kappa =1$. Talay in \cite{Tal84, Tal86} investigated a class
of the second order approximations for diffusion processes. For It\^{o}
processes with jump components, Mikulevi\v{c}ius \& Platen showed the
first-order convergence in the case in which the coefficient functions
possess fourth-order continuous derivatives (see~\cite{MiP88}). In Platen
and Kloeden \& Platen (see \cite{KlP00, Pla99}), not only Euler but also
higher order approximations were studied as well. Protter and Talay in \cite%
{PrT97} considered the weak Euler approximation for\ 
\begin{equation}
X_{t}=X_{0}+\int_{0}^{t}C(X_{s-})dZ_{s},t\in \lbrack 0,T],  \label{pt}
\end{equation}%
where $Z_{t}=(Z_{t}^{1},\ldots ,Z_{t}^{m})$ is a L\'{e}vy process and $%
C=(C^{ij})_{1\leq i\leq d,1\leq j\leq m}$ is a measurable and bounded
function. They\ showed the order of convergence\ $\kappa =1,$\ provided that 
$c$ and $g$\ are smooth and the L\'{e}vy measure of $Z$\ has finite moments
of sufficiently high order. Because of that, the main theorems in \cite%
{PrT97} do not apply to (\ref{exf1}). On the other hand, (\ref{one}) with
non-degenerate $c(x),x\in \mathbf{R}^{d},$ do not cover (\ref{pt}) which can
degenerate completely.

In general, the coefficients and the test function $g$\ do not always have
the smoothness properties assumed in the papers cited above. Mikulevi\v{c}%
ius \& Platen (see \cite{MiP911}) proved that there is still some order of
convergence of the weak Euler approximation for non-degenerate diffusion
processes ((\ref{exf1}) with $\alpha =2$) under H\"{o}lder conditions on the
coefficients and $g$. In Kubilius \& Platen \cite{KuP01}, Platen \&
Bruti-Liberati \cite{plalib} a weak Euler approximation was considered in
the case of a non-degenerate diffusion processes with a finite number of
jumps in finite time intervals.  

This paper is a follow-up to \cite{MiZ10}, where $X_{t}$ was a Markov It\^{o}
process solving a martingale problem. In this paper, we derive the rate of
convergence for (\ref{one}) under $\beta $-H\"{o}lder conditions on the
coefficients. As in \cite{MiZ10} (see \cite{Tal84} as well), we use the
solution to the backward Kolmogorov equation associated with $X_{t}$ and the
one-step estimates derived in \cite{MiZ10}.

In the following Section 2, we introduce assumptions and state the main
result. In Section 3, we present the essential technical results. The main
theorem is proved in Section 4.


\section{Notation and Main Result}



\subsection{Notation}

Denote $H=[0,T]\times \mathbf{R}^{d}$, $\mathbf{N}=\{0,1,2,\ldots \}$, $%
\mathbf{R}_{0}^{d}=\mathbf{R}^{d}\backslash \{0\}$. For $x,y\in \mathbf{R}%
^{d}$, write $(x,y)=\sum_{i=1}^{d}x_{i}y_{i}$, $|x|=\sqrt{(x,x)}$ and $%
|B|=\sum_{i=1}^{d}|B^{ii}|,B\in \mathbf{R}^{d\times d}.$

Let $S^{d-1}$ denote the unit sphere in $\mathbf{R}^{d}$, with $\mu _{d-1}$
being the Lebesgue measure on it.

$C_{b}^{\infty }(H)$ is the set of all functions $u$ on $H$ such that for
all $t\in \lbrack 0,T]$ the function $u(t,x)$ is infinitely differentiable
in $x$ and for every multiindex $\gamma \in \mathbf{N}^{d}$, 
\begin{equation*}
\sup_{(t,x)\in H}|\partial _{x}^{\gamma }u(t,x)|<\infty ,
\end{equation*}%
where%
\begin{equation*}
\partial _{x}^{\gamma }u(t,x)=\frac{\partial ^{|\gamma |}}{\partial ^{\gamma
_{1}}x_{1}\dots \partial ^{\gamma _{d}}x_{d}}u(t,x).
\end{equation*}%
$C_{0}^{\infty }(G)$ is the set of all infinitely differentiable functions
on an open set $G\subseteq \mathbf{R}^{d}$ with compact support. $\mathcal{S}%
(\mathbf{R}^{d})$ is the Schwatz space of rapidly decaying smooth functions.

Denote 
\begin{eqnarray*}
\partial _{t}u(t,x) &=&\frac{\partial }{\partial t}u(t,x), \\
\partial _{i}u(t,x) &=&\frac{\partial }{\partial x_{i}}u(t,x),i=1,\dots ,d,
\\
\partial _{ij}^{2}u(t,x) &=&\frac{\partial ^{2}}{\partial x_{i}x_{j}}%
u(t,x),i,j=1,\dots ,d, \\
\partial _{x}u(t,x) &=&\nabla u(t,x)=\big( \partial _{1}u(t,x),\dots
,\partial _{d}u(t,x)\big) , \\
\partial ^{k}u(t,x) &=&\big( \partial ^{\gamma }u(t,x)\big) _{|\gamma
|=k},k\in \mathbf{N}\text{.}
\end{eqnarray*}

For $\alpha \in (0,2)$, write 
\begin{eqnarray*}
|\partial |^{\alpha }v(x) &=&-\mathcal{F}^{-1}[|\xi |^{\alpha }\mathcal{F}%
v(\xi )](x), \\
|\partial |^{2}v(x) &=&\Delta v(x)=\sum_{i=1}^{d}\partial _{ii}^{2}v(x).
\end{eqnarray*}%
where $\mathcal{F}$ denotes the Fourier transform with respect to $x\in 
\mathbf{R}^{d}$ and $\mathcal{F}^{-1}$ is the inverse Fourier transform,
i.e., 
\begin{equation*}
\mathcal{F}v(\xi )=\int_{\mathbf{R}^{d}}\,\mathrm{e}^{-i(\xi
,x)}u(x)dx,\quad \mathcal{F}^{-1}v(x)=\frac{1}{(2\pi )^{d}}\int_{\mathbf{R}%
^{d}}\,\mathrm{e}^{i(\xi ,x)}v(\xi )d\xi .
\end{equation*}

$C=C(\cdot ,\ldots ,\cdot )$ denotes constants depending only on quantities
appearing in parentheses. In a given context the same letter is (generally)
used to denote different constants depending on the same set of arguments.


\subsection{Assumptions and Main Result}

Assume $m_{\alpha }(x,y)=|h_{\alpha }(x,y)|^{\alpha },x,y\in \mathbf{R}%
^{d},\alpha \in (0,2),$ and its partial derivatives $\partial _{y}^{\gamma
}m_{\alpha }(x,y),|\gamma |\leq d_{0}=\big[ \frac{d}{2}\big] +1$ are
continuous in $(x,y).$ Moreover, $m_{\alpha }(x,y)$ is homogeneous in $y$
with index zero, and $m_{1}(x,y)$ is symmetric in $y$: $%
m_{1}(x,-y)=m_{1}(x,y),x\in \mathbf{R}^{d},y\in S^{d-1}$.

For $\beta =[\beta ]^{-}+\big\{ \beta \big\} ^{+}>0$, where $[\beta ]^{-}\in 
\mathbf{N}$ and $\big\{ \beta \big\} ^{+}\in (0,1]$, let $C^{\beta }(H)$
denote the space of measurable functions $u$ on $H$ such that the norm 
\begin{eqnarray*}
|u|_{\beta } &=&\sum_{|\gamma |\leq \lbrack \beta ]^{-}}\sup_{(t,x)\in
H}|\partial _{x}^{\gamma }u(t,x)|+\sup_{|\gamma |=[\beta ]^{-},t,x\neq 
\tilde{x}}\frac{|\partial _{x}^{\gamma }u(t,x)-\partial _{x}^{\gamma }u(t,%
\tilde{x})|}{|x-\tilde{x}|^{\{\beta \}^{+}}}, \\
\mbox{ if }\{\beta \}^{+} &\in &(0,1), \\
|u|_{\beta } &=&\sum_{|\gamma |\leq \lbrack \beta ]^{-}}\sup_{(t,x)\in
H}|\partial _{x}^{\gamma }u(t,x)| \\
&&+\sup_{|\gamma |=[\beta ]^{-},t,x,h\neq 0}\frac{|\partial _{x}^{\gamma
}u(t,x+h)+\partial _{x}^{\gamma }u(t,x-h)-2\partial _{x}^{\gamma }u(t,x)|}{%
|h|^{\{\beta \}^{+}}}, \\
\mbox{ if }\{\beta \}^{+} &=&1,
\end{eqnarray*}%
is finite. Accordingly, $C^{\beta }(\mathbf{R}^{d})$ denotes the
corresponding space of functions on $\mathbf{R}^{d}$.~The classes $C^{\beta
} $ are H\"{o}lder-Zygmund spaces: they coincide with H\"{o}lder spaces if $%
\beta \notin \mathbf{N}$ (see 1.2.2 of \cite{Tri92}).

Define for $\beta =[\beta ]+\{\beta \}>0$ with $[\beta ]\in \mathbf{N}$, $%
\{\beta \}\in (0,1)$,%
\begin{eqnarray}
M_{\beta }^{(\alpha )} &=&\mathbf{1}_{\{\alpha \in (0,2)\}}|c|_{\beta }+%
\mathbf{1}_{\{\alpha \in \lbrack 1,2]\}}|a|_{\beta }+\mathbf{1}_{\{\alpha
=2\}}|B|_{\beta }  \notag \\
&&+\mathbf{1}_{\{\alpha \in (0,2)\}}\sup_{|\gamma |\leq
d_{0},|y|=1}|\partial _{y}^{\gamma }m^{(\alpha )}(\cdot ,y)|_{\beta }.
\end{eqnarray}

We make the following assumptions.

\textbf{A1 }(i)There is a constant $\mu >0$ such that for all $x\in \mathbf{R%
}^{d}$ and $|\xi |=1,$ 
\begin{eqnarray}
(B(x)\xi ,\xi ) &\geq &\mu ,\mbox{ if }\alpha =2,  \notag \\
\int_{S^{d-1}}|(w,\xi )|^{\alpha }m_{\alpha }(x,w)d\xi &\geq &\mu ,%
\mbox{
if }\alpha \in (0,2),  \label{eqn:bnd_coeffients}
\end{eqnarray}%
where $B(x)=b(x)^{\ast }b(x),x\in \mathbf{R}^{d}$;

(ii)\textbf{\ }It holds that\textbf{\ \ }%
\begin{equation*}
\lim_{n\rightarrow \infty }\sup_{x}\int_{U_{n}}|l_{\alpha }(x,\upsilon
)|^{\alpha }\pi (d\upsilon )=0,\text{ if }\alpha \in (0,2].
\end{equation*}

\textbf{A2(}$\beta )$ It satisfies that $M_{\beta }^{(\alpha )}<\infty
,\inf_{x}|\det c(x)|>0,$ and 
\begin{eqnarray*}
&&\int \big\{\mathbf{1}_{U_{1}}(\upsilon )\big[|l_{\alpha }(x,\upsilon
)|^{\alpha } + \mathbf{1}_{\{\beta \geq 1\}}\sum_{j=1}^{[\beta ]}(|\partial
^{j}l_{\alpha }(x,\upsilon )|^{\alpha }+|\partial ^{j}l_{\alpha }(x,\upsilon
)|^{\frac{[\beta ]}{j}\vee \alpha })\big] \\
&&+\mathbf{1}_{U_{1}^{c}}(\upsilon )\big[|l_{\alpha }(x,\upsilon )|^{\alpha
\wedge 1}\wedge 1+\mathbf{1}_{\{\beta \geq 1\}}\sum_{j=1}^{[\beta
]}(|\partial ^{j}l_{\alpha }(x,\upsilon )|+|\partial ^{j}l^{(\alpha
)}(x,\upsilon )|^{\frac{[\beta ]}{j}})\big]\big\}\pi (d\upsilon ) \\
&\leq &K,
\end{eqnarray*}

\textbf{A3}($\beta )$ For all $x,x^{\prime }\in \mathbf{R}^{d},$%
\begin{eqnarray*}
\int_{U_{1}}[|l_{\alpha }(x,\upsilon )-l_{\alpha }(x^{\prime },\upsilon
)|^{\alpha }+|\partial ^{\lbrack \beta ]}l_{\alpha }(x,\upsilon )-\partial
^{\lbrack \beta ]}l_{\alpha }(x^{\prime },\upsilon )|^{\alpha }]\pi
(d\upsilon ) \leq C|x-x^{\prime }|^{\alpha \beta },\alpha \in \lbrack 1,2],
\end{eqnarray*}%
There exists $\beta ^{\prime }$ such that $\beta \le \alpha +\beta ^{\prime
} < \alpha +\beta$ and for all $x,x^{\prime }\in \mathbf{R}^{d},$%
\begin{eqnarray*}
&&\mathbf{1}_{\{\beta \geq 1\}}\int_{U_{1}}(|l_{\alpha }(x,\upsilon
)-l_{\alpha }(x^{\prime },\upsilon )|^{(\alpha +\beta ^{\prime }-[\beta
])\wedge 1}\wedge 1) \\
& & \times \sum_{j=1}^{[\beta ]}(|\partial ^{j}l_{\alpha }(x,\upsilon
)|^{\alpha \vee 1}+|\partial ^{j}l_{\alpha }(x,\upsilon )|^{\frac{[\beta ]}{j%
}\vee \alpha })\pi (d\upsilon ) \\
&\leq &C|x-x^{\prime }|^{\beta -[\beta ]}, \\
&&\int_{U_{1}^{c}}[|l_{\alpha }(x,\upsilon )-l_{\alpha }(x^{\prime
},\upsilon )|^{(\alpha +\beta ^{\prime }-[\beta ])\wedge 1}\wedge 1] \\
& & \times \big[1+\mathbf{1}_{\{\beta \geq 1\}}\sum_{j=1}^{[\beta
]}(|\partial ^{j}l_{\alpha }(x,\upsilon )|+|\partial ^{j}l_{\alpha
}(x,\upsilon )|^{\frac{[\beta ]}{j}})\big]\pi (d\upsilon ) \\
&\leq &C|x-x^{\prime }|^{\beta -[\beta ]}.
\end{eqnarray*}

\textbf{A4}($\beta )$ For $\beta \geq 1,x,x^{\prime }\in \mathbf{R}^{d},$%
\begin{eqnarray*}
&&\mathbf{1}_{\{\beta \geq 1\}}\sum_{j=1}^{[\beta ]}
\int_{U_{1}^{c}}|\partial ^{j}l_{\alpha }(x,\upsilon )-\partial
^{j}l_{\alpha }(x^{\prime },\upsilon )|\pi (d\upsilon ) \\
& & +\Big( \int_{U_{1}^{c}}|\partial ^{j}l_{\alpha }(x,\upsilon )-\partial
^{j}l_{\alpha }(x^{\prime },\upsilon )|^{\frac{[\beta ]}{j}}\pi (d\upsilon) %
\Big) ^{\frac{j}{[\beta ]}} \\
&\leq &C|x-x^{\prime }|^{\beta -[\beta ]},
\end{eqnarray*}%
and%
\begin{eqnarray*}
&\mathbf{1}_{\{\beta \geq 1\}}&\sum_{j=1}^{[\beta ]}\Big( %
\int_{U_{1}}|\partial ^{j}l_{\alpha }(x,\upsilon )-\partial ^{j}l_{\alpha
}(x^{\prime },\upsilon )|^{\alpha \vee 1}\pi (d\upsilon )\Big) ^{\frac{1}{%
\alpha }\wedge 1} \\
&& \\
&&+\sum_{j=1}^{[\beta ]}\Big( \int_{U_{1}}|\partial ^{j}l_{\alpha
}(x,\upsilon )-\partial ^{j}l_{\alpha }(x^{\prime },\upsilon )|^{\frac{%
[\beta ]}{j}\vee \alpha }\pi (d\upsilon )\Big) ^{\frac{j}{[\beta ]}\wedge 
\frac{1}{\alpha }} \\
&\leq &C|x-x^{\prime }|^{\beta -[\beta ]}.
\end{eqnarray*}


The main result of this paper is the following statement.

\begin{theorem}
\label{thm:main} Let $\alpha \in (0,2]$, $\beta >0,\beta \notin \mathbf{N}$.
Assume \textup{A1-A4}$\QTR{up}{(}\beta )$ hold. Then there exists a constant 
$C$ such that for all $g\in C^{\alpha +\beta }(\mathbf{R}^{d}),f\in C^{\beta
}(\mathbf{R}^{d})$%
\begin{eqnarray}
|\mathbf{E}g(Y_{T})-\mathbf{E}g(X_{T})| &\leq &C|g|_{\alpha +\beta }\delta
^{\kappa (\alpha ,\beta )},  \label{mt1} \\
|\mathbf{E}\int_{0}^{T}f(Y_{\tau _{i_{s}}})ds-\mathbf{E}%
\int_{0}^{T}f(X_{s})ds| &\leq &C|f|_{\beta }\delta ^{\kappa (\alpha ,\beta
)},  \notag
\end{eqnarray}%
where%
\begin{equation*}
\kappa (\alpha ,\beta )=\left\{ 
\begin{array}{cl}
\frac{\beta }{\alpha }, & \beta <\alpha , \\ 
1, & \beta >\alpha .%
\end{array}%
\right.
\end{equation*}
\end{theorem}

\begin{remark}
\label{rema1}

\begin{enumerate}
\item The second condition of \textup{A1(i)} holds with some constant $\mu
>0 $ if, for example, there is a Borel set $\Gamma \subseteq S^{d-1}$ such
that $\mu _{d-1}(\Gamma )>0$ and $\inf_{x\in \mathbf{R}^{d},w\in \Gamma
}m_{\alpha }(x,w)>0$.

\item The assumptions \textup{A1}-\textup{A4($\beta )$} guarantee that the
solution to the backward Kolmogorov equation associated with $X_{t}$ is $%
(\alpha +\beta )$-H\"{o}lder. If $\alpha =2$ and the operator is differential%
$,$ the assumptions imposed are standard classical. The regularity of the
solution determines the rate of convergence of a weak Euler approximation.
\end{enumerate}
\end{remark}

\subsection{SDEs driven by L\'{e}vy processes}

Let $Z^{0}=Z_{t}^{0},t\in \lbrack 0,T],$ be a standard $d$-dimensional
spherically-symmetric $\alpha $-stable process $($see $(\ref{st})$ for the
definition$)$ with jump measure $p_{0}$ and martingale measure $q_{0}$, and
let $Z_{t}=(Z_{t}^{1},\ldots ,Z_{t}^{m})$ be an independent $m$-dimensional L%
\'{e}vy process defined by%
\begin{eqnarray}
Z_{t} &=&\int_{0}^{t}\int g_{\alpha }(\upsilon )p(ds,d\upsilon ),\alpha \in
(0,1),  \notag \\
Z_{t} &=&\int_{0}^{t}\int_{U_{1}}g_{\alpha }(\upsilon )q(ds,d\upsilon
)+\int_{0}^{t}\int_{U_{1}^{c}}g_{\alpha }(\upsilon )p(ds,d\upsilon ),\alpha
=1, \\
Z_{t} &=&\int_{0}^{t}\int g_{\alpha }(\upsilon )q(ds,d\upsilon ),\alpha \in
(1,2],  \notag
\end{eqnarray}%
where $g_{\alpha }=(g_{\alpha }^{i})_{1\leq i\leq m}$ is a measurable
function on $U$ and%
\begin{equation*}
\int_{U_{1}}|g_{\alpha }(\upsilon )|^{\alpha }\pi (d\upsilon )+\mathbf{1}%
_{\{\alpha \in (1,2)\}}\int_{U_{1}^{c}}|g_{\alpha }(\upsilon )|\pi
(d\upsilon )<\infty .
\end{equation*}

Consider for $t\in \lbrack 0,T]$, 
\begin{eqnarray}
X_{t}
&=&X_{0}+\int_{0}^{t}c(X_{s-})dZ_{s}^{0}+\int_{0}^{t}C(X_{s-})dZ_{s},t\in
\lbrack 0,T],\alpha \in (0,2),  \label{exf1} \\
X_{t}
&=&X_{0}+\int_{0}^{t}a_{2}(X_{s})ds+\int_{0}^{t}b(X_{s})dW_{s}+%
\int_{0}^{t}C(X_{s-})dZ_{s},\alpha =2,  \notag
\end{eqnarray}%
where $c(x)=(c^{ij}(x))_{1\leq i,j\leq d},C(x)=(C^{ij}(x))_{1\leq i\leq
d,1\leq j\leq m},x\in \mathbf{R}^{d},$ are measurable and bounded. Assume
that $c$ is non-degenerate with $\inf_{x}\det |c(x)|>0.$ Obviously, $(\ref%
{exf1})$ can be rewritten as%
\begin{eqnarray*}
X_{t} &=&X_{0}+\int_{0}^{t}\int_{\mathbf{R}_{0}^{d}}c(X_{s-})yp_{0}(ds,dy) \\
&&+\int_{0}^{t}\int C(X_{s-})g_{\alpha }(\upsilon )p(ds,d\upsilon )\text{ if 
}\alpha \in (0,1), \\
X_{t}
&=&X_{0}+\int_{0}^{t}\int_{|y|>1}c(X_{s-})yp_{0}(ds,dy)+\int_{0}^{t}%
\int_{|y|\leq 1}c(X_{s-})yq_{0}(ds,dy) \\
&&+\int_{0}^{t}\int_{U_{1}}C(X_{s-})g_{\alpha }(\upsilon )q(ds,d\upsilon
)+\int_{0}^{t}\int_{U_{1}^{c}}C(X_{s-})g_{\alpha }(\upsilon )p(ds,d\upsilon )%
\text{ if }\alpha =1, \\
X_{t} &=&X_{0}+\int_{0}^{t}\int_{\mathbf{R}_{0}^{d}}c(X_{s-})yq_{0}(ds,dy) \\
&&+\int_{0}^{t}\int_{U_{1}}C(X_{s-})g_{\alpha }(\upsilon )q(ds,d\upsilon
)+\int_{0}^{t}\int_{U_{1}^{c}}C(X_{s-})g_{\alpha }(\upsilon )p(ds,d\upsilon )%
\text{ if }\alpha \in (1,2), \\
X_{t} &=&X_{0}+\int_{0}^{t}a(X_{s})ds+\int_{0}^{t}b(X_{s})dW_{s} \\
&&+\int_{0}^{t}\int_{U_{1}}C(X_{s-})g_{\alpha }(\upsilon )q(ds,d\upsilon
)+\int_{0}^{t}\int_{U_{1}^{c}}C(X_{s-})g_{\alpha }(\upsilon )p(ds,d\upsilon )%
\text{ if }\alpha =2,
\end{eqnarray*}

Applying Theorem \ref{thm:main} to (\ref{exf1}) we obtain easily the
following statements.

\begin{proposition}
\label{cor:main} Let $X_{t},t\in \lbrack 0,T]$ satisfy \textup{(\ref{exf1}), 
}$\inf_{x\in \mathbf{R}^{d}}|\det c(x)|>0,\inf_{x\in \mathbf{R}^{d}}|\det
b(x)|>0$.

For $\alpha \in (0,1)$, we assume $\beta \in (0,1)$,$\alpha +\beta >1$,$%
c^{ij},C^{ij}\in C^{\beta }(\mathbf{R}^{d})$ and%
\begin{equation*}
\int [|g_{\alpha }(\upsilon )|+|g_{\alpha }(\upsilon )|^{\alpha }]\pi
(d\upsilon )<\infty .
\end{equation*}

For $\alpha \in \lbrack 1,2),$ we assume $\beta \neq 1,\beta <2$, $%
a,c^{ij},C^{ij}\in C^{\beta }(\mathbf{R}^{d})$ and%
\begin{equation*}
\int_{U_{1}}[|g_{\alpha }(\upsilon )|^{\alpha }+|g_{\alpha }(\upsilon
)|^{\alpha +[\beta ]}]\pi (d\upsilon )+\int_{U_{1}^{c}}[|g_{\alpha
}(\upsilon )|+|g_{\alpha }(\upsilon )|^{1+[\beta ]}]\pi (d\upsilon )<\infty .
\end{equation*}

For $\alpha =2,$ we assume $\beta <3,\beta \notin \mathbf{N},a,b,C^{ij}\in
C^{\beta }(\mathbf{R}^{d})$ and%
\begin{equation*}
\int_{U_{1}}[|g_{\alpha }(\upsilon )|^{2}+|g_{\alpha }(\upsilon )|^{2+[\beta
]}]\pi (d\upsilon )+\int_{U_{1}^{c}}[|g_{\alpha }(\upsilon )|+|g_{\alpha
}(\upsilon )|^{1+[\beta ]}]\pi (d\upsilon )<\infty .
\end{equation*}%
Then there exists a constant $C$ such that for all $g\in C^{\alpha +\beta }(%
\mathbf{R}^{d}),f\in C^{\beta }(\mathbf{R}^{d})$%
\begin{eqnarray*}
|\mathbf{E}g(Y_{T})-\mathbf{E}g(X_{T})| &\leq &C|g|_{\alpha +\beta }\delta ^{%
\frac{\beta }{\alpha }\wedge 1}, \\
|\mathbf{E}\int_{0}^{T}f(Y_{\tau _{i_{s}}})ds-\mathbf{E}%
\int_{0}^{T}f(X_{s})ds| &\leq &C|f|_{\beta }\delta ^{\frac{\beta }{\alpha }%
\wedge 1}.
\end{eqnarray*}
\end{proposition}

\begin{remark}
Proposition \ref{cor:main} improves the rate of convergence for diffusion
processes in \textup{\cite{MiP911}} with $\beta \in (1,2)$. Under the
assumption of Proposition \ref{cor:main} $($with $\alpha =2,C^{ij}=0)$, it
was derived in \textup{\cite{MiP911}} that the convergence rate is of the
order $\frac{1}{3-\beta }<\kappa (2,\beta )=\frac{\beta }{2}$ if $\beta \in
(1,2)$.
\end{remark}


\section{Backward Kolmogorov Equation}


To determine the form of the backward Kolmogorov equation associated with $%
X_{t}$ in (\ref{one}), we find the compensator of the jump measure of $X$
first.

\begin{lemma}
\label{le1}Let $p^{X}$ be the jump measure of $X_{t}$ in $(\ref{one})$. Then 
\begin{equation*}
q^{X}(dt,dy)=p^{X}(dt,dy)=\tilde{m}_{\alpha }(X_{t-},\frac{y}{|y|})\frac{dy}{%
|y|^{d+\alpha }}dt + \int_{U}\mathbf{1}_{dy}(l_{\alpha }(X_{t-},\upsilon
))\pi (d\upsilon )dt
\end{equation*}%
is a martingale measure, where%
\begin{equation*}
\tilde{m}_{\alpha }(x,\frac{y}{|y|})=\frac{1}{\big\vert \det c(x)\big\vert }%
\frac{1}{|c(x)^{-1}\frac{y}{|y|}|^{d+\alpha }}m_{\alpha }\big( x,\frac{%
c(x)^{-1}\frac{y}{|y|}}{\big\vert c(x)^{-1}\frac{y}{|y|}\big\vert }\big) %
,x\in \mathbf{R}^{d},y\in \mathbf{R}_{0}^{d}.
\end{equation*}
\end{lemma}

\begin{proof}
Since $p_{0}$ and $p$ have no common jumps, for any $t$ and $\Gamma \in 
\mathcal{B}(\mathbf{R}_{0}^{d}),$ 
\begin{eqnarray*}
p^{X}((0,t]\times \Gamma ) &=&\sum_{s\leq t}\mathbf{1}_{\Gamma }(\Delta
X_{t})=\int_{0}^{t}\int_{\mathbf{R}_{0}^{d}}\mathbf{1}_{\Gamma
}(c(X_{s-})h_{\alpha }(X_{s-},\frac{y}{|y|})y)p_{0}(ds,dy) \\
&&+\int_{0}^{t}\int_{U}\mathbf{1}_{\Gamma }(l_{\alpha }(X_{s-},\upsilon
))p(ds,d\upsilon ),
\end{eqnarray*}%
with $\Delta X_{t}=X_{t}-X_{t-},0<t,\Gamma \in \mathcal{B}(\mathbf{R}%
_{0}^{d})$. Passing to polar coordinates and changing the variable of
integration twice%
\begin{eqnarray*}
&&\int_{\mathbf{R}_{0}^{d}}\mathbf{1}_{\Gamma }(c(x)h_{\alpha }(x,\frac{y}{%
|y|})y)\frac{dy}{|y|^{d+\alpha }} \\
&=&\int_{0}^{\infty }\int_{S^{d-1}}\mathbf{1}_{\Gamma }(c(x)h_{\alpha
}(x,w)\rho w)\mu _{d-1}(dw)\frac{d\rho }{\rho ^{1+\alpha }} \\
&=&\int_{0}^{\infty }\int_{S^{d-1}}\mathbf{1}_{\Gamma }(c(x)\rho w)\mu
_{d-1}(dw)\frac{h_{\alpha }(x,w)^{\alpha }d\rho }{\rho ^{1+\alpha }} \\
&=&\int_{\mathbf{R}_{0}^{d}}\mathbf{1}_{\Gamma }(c(x)y)\frac{h_{\alpha }(x,%
\frac{y}{|y|})^{\alpha }dy}{|y|^{d+\alpha }}=\int_{\mathbf{R}_{0}^{d}}%
\mathbf{1}_{\Gamma }(y)\tilde{m}_{\alpha }(x,\frac{y}{|y|})\frac{dy}{%
|y|^{d+\alpha }}.
\end{eqnarray*}%
The statement follows.
\end{proof}

For $u\in C^{\alpha +\beta }(H)$, denote 
\begin{equation*}
A_{y}^{(\alpha )}u(t,x)=u(t,x+y)-u(t,x)-\chi ^{(\alpha )}(y)(\nabla
_{x}u(t,x),y),
\end{equation*}%
where $\chi ^{(\alpha )}(y)=\mathbf{1}_{\{|y|\leq 1\}}\mathbf{1}_{\{\alpha
=1\}}+\mathbf{1}_{\{\alpha \in (1,2)\}}$. Let%
\begin{eqnarray}
\mathcal{A}_{z}^{(\alpha )}u(t,x) &=&\mathbf{1}_{\{\alpha
=1\}}(a_{1}(z),\nabla _{x}u(t,x))+\frac{1}{2}\mathbf{1}_{\{\alpha
=2\}}\sum_{i,j=1}^{d}B^{ij}(z)\partial _{ij}^{2}u(t,x)  \notag  \label{ff3}
\\
&&+\int_{\mathbf{R}_{0}^{d}}A_{y}^{(\alpha )}u(t,x)\tilde{m}_{\alpha }(z,%
\frac{y}{|y|})\frac{dy}{|y|^{d+\alpha }},~x,z\in \mathbf{R}^{d}, \\
\mathcal{A}^{(\alpha )}u(t,x) &=&\mathcal{A}_{x}^{(\alpha )}u(t,x)=\mathcal{A%
}_{z}^{(\alpha )}u(t,x)|_{z=x},x\in \mathbf{R}^{d}.  \notag
\end{eqnarray}%
Let 
\begin{eqnarray}
\mathcal{B}_{z}^{(\alpha )}u(t,x) &=&\mathbf{1}_{\{\alpha \in
(1,2]\}}(a(z),\nabla _{x}u(t,x))+\int_{U}\big[u(t,x+l_{\alpha }(z,\upsilon
))-u(t,x)  \notag  \label{ff10} \\
&&-\mathbf{1}_{\{\alpha \in (1,2]\}}\mathbf{1}_{U_{1}}(\upsilon )(\nabla
_{x}u(t,x),l_{\alpha }(z,\upsilon ))\big]\pi (d\upsilon ), \\
\mathcal{B}^{(\alpha )}u(t,x) &=&\mathcal{B}_{x}^{(\alpha )}u(t,x)=\mathcal{B%
}_{z}^{(\alpha )}u(t,x)|_{z=x},x\in \mathbf{R}^{d}.  \notag
\end{eqnarray}

\begin{remark}
\label{lrenew1}Under assumptions \textup{A1-A4($\beta )$}, for any $\beta >0$%
, there exists a unique weak solution to equation \textup{(\ref{one})} and
for every $u\in C^{\alpha +\beta }(\mathbf{R}^{d}),$ the stochastic process 
\begin{equation}
u(X_{t})-\int_{0}^{t}(\mathcal{A}^{(\alpha )}+\mathcal{B}^{(\alpha
)})u(X_{s})ds  \label{ff25}
\end{equation}%
is a martingale $($see~\textup{\cite{MiP923}}$)$. The operator $\mathcal{L}%
^{(\alpha )}=\mathcal{A}^{(\alpha )}+\mathcal{B}^{(\alpha )}$ is the
generator of $X_{t}$ defined in \textup{(\ref{one})}; $\mathcal{A}^{(\alpha
)}$ is the principal part of $\mathcal{L}^{(\alpha )}$ and $\mathcal{B}%
^{(\alpha )}$ is the lower order or subordinated part of $\mathcal{L}%
^{(\alpha )}$.
\end{remark}

\begin{remark}
If $m^{(\alpha )}=1$, $(B^{ij})=I$ $(d\times d$-identity matrix$)$, $%
a_{1}(z)=0$, then $\mathcal{A}^{(\alpha )}$ is the generator of a standard
spherically-symmetric $\alpha $-stable process 
\begin{eqnarray}
Z_{t} &=&\int_{0}^{t}\int yq^{Z}(ds,dy),\alpha \in (1,2),  \notag \\
Z_{t} &=&\int_{0}^{t}\int_{|y|\leq
1}yq^{Z}(ds,dy)+\int_{0}^{t}\int_{|y|>1}yp^{Z}(ds,dy),\alpha =1,  \label{st}
\\
Z_{t} &=&\int_{0}^{t}\int yp^{Z}(ds,dy),\alpha \in (0,1),  \notag
\end{eqnarray}%
where $p^{Z}(ds,dy)$ is the jump measure of $Z$ and 
\begin{equation*}
q^{Z}(ds,dy)=p^{Z}(ds,dy)-\frac{dyds}{|y|^{d+\alpha }}
\end{equation*}%
is the martingale measure; $Z_{t}$ is the standard Wiener process if $\alpha
=2.$
\end{remark}

We consider in H\"{o}lder-Zygmund spaces the backward Kolmogorov equation
associated with $X_{t}$ (see \cite{Tal84}, \cite{MiZ10}):%
\begin{eqnarray}
\big( \partial _{t}+\mathcal{A}_{x}^{(\alpha )}+\mathcal{B}_{x}^{(\alpha
)}-\lambda \big) v(t,x) &=&f(t,x),  \notag \\
v(T,x) &=&0  \label{eqn:cauchy_prf}
\end{eqnarray}%
with $\lambda \geq 0$. The regularity of its solution is essential for the
one step estimate that determines the rate of convergence.

\begin{definition}
\label{def1}Let $f$ be a bounded measurable function on $\mathbf{R}^{d}$. We
say that $u\in C^{\alpha +\beta }(H)$ is a solution to $(\ref{eqn:cauchy_prf}%
)$, if for each $(t,x)\in H$,%
\begin{equation}
u(t,x)=\int_{0}^{t}\big[\big( \partial _{t}+\mathcal{L}^{(\alpha )}-\lambda %
\big) u(s,x)-\lambda u(s,x)+f(s,x)\big]ds.  \label{defs}
\end{equation}
\end{definition}

\begin{theorem}
\label{thm:StoCP} Let $\alpha \in (0,2]$, $\beta >0,\beta \notin \mathbf{N}$%
, and $f\in C^{\beta }(H)$. Assume \textup{A1}-\textup{A4$(\beta )$} hold.
Then there exists a unique solution $v\in C^{\alpha +\beta }(H)$ to \textup{(%
\ref{eqn:cauchy_prf})}. Moreover, there is a constant $C$ independent of $f$
such that%
\begin{equation*}
|u|_{\alpha +\beta }\leq C|f|_{\beta }.
\end{equation*}
\end{theorem}

An immediate consequence of this theorem is the following statement.

\begin{corollary}
\label{lcornew1}Let $\alpha \in (0,2]$ and $\beta >0,\beta \notin \mathbf{N}$%
. Assume \textup{A1-A4$(\beta )$} hold, $f\in C^{\beta }(H)$, and $g\in
C^{\alpha +\beta }(\mathbf{R}^{d})$. Then there exists a unique solution $%
v\in C^{\alpha +\beta }(H)$ to the Cauchy problem 
\begin{eqnarray}
\big( \partial _{t}+\mathcal{A}_{x}^{(\alpha )}+\mathcal{B}_{x}^{(\alpha )}%
\big) v(t,x) &=&f(x),  \label{maf8} \\
v(T,x) &=&g(x),  \notag
\end{eqnarray}%
and $|v|_{\alpha +\beta }\leq C(|f|_{\beta }+|g|_{\alpha +\beta })$ with a
constant $C$ independent of $f$ and $g$.
\end{corollary}

To prove Theorem \ref{thm:StoCP} and Corollary \ref{lcornew1}, we first
derive H\"{o}lder norm estimates of $\mathcal{A}^{(\alpha )}f$ and $\mathcal{%
B}^{(\alpha )}f$, $f\in C^{\alpha +\beta }(\mathbf{R}^{d})$, $\beta >0$, and
an auxiliary lemma about uniform convergence of H\"{o}lder functions.


\subsection{Kolmogorov Equation with Constant Coefficients}

Let $B=(B^{ij})_{1\leq i,j\leq d}$ be a non-negative definite non-degenerate
matrix. Let $r_{\alpha }(y)$ be homogeneous with index zero and
differentiable in $y$ up to the order $d_{0}=[d/2]+1$ and%
\begin{equation*}
\int_{S^{d-1}}wr_{1}(w)\mu _{d-1}(dw)=0,r_{2}=0.
\end{equation*}%
Let 
\begin{eqnarray*}
A_{\alpha }^{0}u(x) &=& \mathbf{1}_{\{\alpha=2\}} B^{ij}\partial
_{ij}^{2}u(x) + \mathbf{1}_{\{\alpha =1\}} a_{1}^{i}\partial _{i}u(x)+\int_{%
\mathbf{R}^{d}}\big[u(x+y)-u(x) \\
&&-(\mathbf{1}_{\{|y|\leq 1\}} \mathbf{1}_{\{\alpha =1\}}+\mathbf{1}%
_{\{1<\alpha <2\}})(\nabla u(x),y)\big]r_{\alpha }(y)\frac{dy}{|y|^{d+\alpha
}}.
\end{eqnarray*}%
In terms of Fourier transform, 
\begin{equation*}
A_{\alpha }^{0}u(x)=\mathcal{F}^{-1}\big[ \psi _{\alpha }^{0}(\xi )\mathcal{F%
}u(\xi )\big] (x),
\end{equation*}%
where%
\begin{eqnarray*}
\psi _{\alpha }^{0}(\xi ) &=& -N\int_{S^{d-1}}|(w,\xi )|^{\alpha }\big[1-i%
\big(\mathbf{1}_{\{\alpha \neq 1\}}\tan \frac{\alpha \pi }{2}\mathrm{sgn}%
(w,\xi ) \\
&&- \frac{2}{\pi }\mathbf{1}_{\{\alpha =1\}}\mathrm{sgn}(w,\xi )\ln |(w,\xi
)|\big)\big]r_{\alpha }(w)\mu _{d-1}(dw) \\
&&-i\mathbf{1}_{\{\alpha =1\}}(a_{1},\xi )-\mathbf{1}_{\{\alpha =2\}}(B\xi
,\xi ),
\end{eqnarray*}%
where $a_{1}\in \mathbf{R}^{d}$. \ We will need the following assumptions.

\textbf{B. }(i) There is a constant $\mu >0$ such that for all $|\xi |=1,$ 
\begin{eqnarray*}
(B\xi ,\xi ) &\geq &\mu ,\mbox{ if }\alpha =2, \\
\int_{S^{d-1}}|(w,\xi )|^{\alpha }r_{\alpha }(w)\mu _{d-1}(dw) &\geq &\mu ,%
\mbox{
if }\alpha \in (0,2);
\end{eqnarray*}

(ii) There is a constant $K$ such that%
\begin{equation*}
|a_{1}|+|B|+\sup_{|\gamma |\leq d_{0},y\in \mathbf{R}^{d}}|\partial ^{\gamma
}r_{\alpha }(y)|\leq K.
\end{equation*}

Consider for $\lambda \geq 0$ the Cauchy problem 
\begin{eqnarray}  \label{n1}
\begin{array}{rcll}
\partial_{t}u(t,x) & = & A_{\alpha }^{0}u(t,x)-\lambda u(t,x)+f(x), & 
(t,x)\in H \\ 
u(0,x) & = & 0, & x\in \mathbf{R}^{d}.%
\end{array}%
\end{eqnarray}

We will solve this equation for $f\in C_{b}^{\infty }(\mathbf{R}^{d})$ and
pass to the limit. The following approximation statement is needed.

\begin{lemma}
\label{pagle1}Let $\beta >0,f\in C^{\beta }(\mathbf{R}^{d})$. Then there is
a sequence $f_{n}\in C_{b}^{\infty }(\mathbf{R}^{d})$ such that%
\begin{equation*}
|f_{n}|_{\beta }\leq 2|f|_{\beta },|f|_{\beta }\leq \lim
\inf_{n}|f_{n}|_{\beta },
\end{equation*}%
and for any $0<\beta ^{\prime }<\beta $, 
\begin{equation*}
|f_{n}-f|_{\beta ^{\prime }}\rightarrow 0,\mbox{as $n\rightarrow \infty$}.
\end{equation*}
\end{lemma}

\begin{proof}
By Lemma 6.1.7 in \cite{BeL76}, there exists a function $\phi \in
C_{0}^{\infty }(\mathbf{R}^{d})$ such that supp$\phi (\xi )=\{\xi :\frac{1}{2%
}\leq |\xi |\leq 2\},\phi (\xi )>0$, if $2^{-1}<|\xi |<2$, and%
\begin{equation*}
\sum_{j=-\infty }^{\infty }\phi (2^{-j}\xi )=1\text{, if }\xi \neq 0.
\end{equation*}%
Define functions $\varphi _{k}\in \mathcal{S}(\mathbf{R}^{d}),k=1,2,\ldots $
by%
\begin{equation}
\mathcal{F\varphi }_{k}=\phi (2^{-k}\xi ),  \label{fu1}
\end{equation}%
and $\varphi _{0}\in \mathcal{S}(\mathbf{R}^{d})$ by%
\begin{equation}
\mathcal{F}\varphi _{0}=1-\sum_{k\geq 1}\varphi _{k}(\xi ).  \label{fu2}
\end{equation}%
We will use on $C^{\beta }(\mathbf{R}^{d})$ an equivalent norm (see see
2.3.8 and 2.3.1 in \cite{Tri83}%
\begin{equation*}
|f|_{\beta ;\infty \infty }=\sup_{k\geq 0,x\in \mathbf{R}^{d}}2^{\beta
k}|\varphi _{k}\ast f(x)|.
\end{equation*}%
Obviously, $f_{n}\in C_{b}^{\infty }(\mathbf{R}^{d})$. Let 
\begin{equation*}
f_{n}=\sum_{k=0}^{n}\varphi _{k}\ast f,n\geq 1.
\end{equation*}%
Since%
\begin{equation*}
\varphi _{k}=\sum_{l=-1}^{1}\varphi _{k+l}\ast \varphi _{k},\varphi
_{0}=(\varphi _{0}+\varphi _{1})\ast \varphi _{0},
\end{equation*}%
we have for large $n,$%
\begin{eqnarray*}
f_{n}\ast \varphi _{k} &=&f\ast \varphi _{k},k<n, \\
f_{n}\ast \varphi _{n} &=&f\ast \varphi _{n}-f\ast \varphi _{n+1}\ast
\varphi _{n}, \\
f_{n}\ast \varphi _{n+1} &=&f\ast \varphi _{n}\ast \varphi _{n+1},f_{n}\ast
\varphi _{k}=0,k>n+1,
\end{eqnarray*}%
and the statement follows.%
\begin{equation*}
|f_{n}|_{\beta }\leq 2|f|_{\alpha ,\beta },|f|_{\beta }\leq \lim
\inf_{n}|f_{n}|_{\beta },
\end{equation*}%
and for any $0<\beta ^{\prime }<\beta $%
\begin{eqnarray*}
|f_{n}-f|_{\beta ^{\prime };\infty \infty } &\leq &2\sup_{k\geq
n}\sup_{x}2^{\beta ^{\prime }k}|\varphi _{k}\ast f(x)| \\
&\leq &2\cdot 2^{-(\beta -\beta ^{\prime })n}\sup_{k\geq n}\sup_{x}2^{\beta
k}|\varphi _{k}\ast f(x)|\rightarrow 0
\end{eqnarray*}%
as $n\rightarrow \infty .$
\end{proof}

\begin{proposition}
\label{prop1}Let $\alpha \in (0,2],\ \beta >0,\beta \notin \mathbf{N},f\in
C^{\beta }(\mathbf{R}^{d})$. Assume \textup{B} holds. Then there is a unique
solution $u\in C^{\alpha +\beta }(H)$ to $(\ref{n1})$ and 
\begin{equation*}
|u|_{\alpha +\beta }\leqslant C|f|_{\beta },
\end{equation*}%
where the constant $C$ depends only on $\alpha ,\ \beta ,\ T,\ d$, $\mu ,\
K. $ Moreover,%
\begin{equation*}
|u|_{\beta }\leq C(\alpha ,d)(\lambda ^{-1}\wedge T)|f|_{\beta },
\end{equation*}%
and there is a constant $C$ such that for all $s\leq t\leq T,$ 
\begin{equation*}
|u(t,\cdot )-u(s,\cdot )|_{\alpha /2+\beta }\leq C(t-s)^{1/2}|f|_{\beta }.
\end{equation*}
\end{proposition}

\begin{proof}
By Lemma \ref{pagle1} there is a sequence $f_{n}\in C_{b}^{\infty }(\mathbf{R%
}^{d})$ such that 
\begin{equation*}
|f_{n}|_{\beta }\leq 2|f|_{\beta },|f|_{\beta }\leq \lim
\inf_{n}|f_{n}|_{\beta },
\end{equation*}%
and for any $\beta ^{\prime }<\beta $ 
\begin{equation}
|f_{n}-f|_{\beta ^{\prime }}\rightarrow 0,\mbox{as $n \rightarrow \infty$.}
\label{nulis}
\end{equation}%
Then, by \ Lemma 7 in \cite{MiP09}\ for each $n$ there is a unique $u_{n}\in
C_{b}^{\infty }(H)$ solving (\ref{n1}). Moreover,%
\begin{equation*}
u_{n}(t,x)=R_{\lambda }f_{n}(t,x)=\int_{0}^{t}G_{s,t}^{\lambda }\ast
f_{n}(x)ds,
\end{equation*}%
where%
\begin{equation*}
G_{s,t}^{\lambda }(x)=\mathcal{F}^{-1}\exp \Big\{\int_{s}^{t}(\psi _{\alpha
}^{0}(r,\xi )-\lambda )dr\Big\},0\leq s\leq t\leq T.
\end{equation*}%
Since for any $k\leq \lbrack \beta ],$%
\begin{equation*}
\partial ^{k}u_{n}(t,x)=\int_{0}^{t}G_{s,t}^{\lambda }\ast \partial
^{k}f_{n}(x)ds=R_{\lambda }\big(\partial ^{k}f_{n}\big),
\end{equation*}%
it follows by Lemma 17 in \cite{MiP09} that for every $\beta ^{\prime }\in
([\beta ],\beta ]$ there is a constant $C$ depending only on $\alpha ,\
\beta ^{\prime },\ p,\ T,\ d$, $\mu ,\ K$ such that 
\begin{equation}
|\partial ^{k}u_{n}|_{\alpha +\beta ^{\prime }-[\beta ]}\leqslant C|\partial
^{k}f_{n}|_{\beta ^{\prime }-[\beta ]}  \label{vienas}
\end{equation}%
for all $k\leq \lbrack \beta ]$. Moreover,%
\begin{equation*}
|\partial ^{k}u_{n}|_{\beta ^{\prime }-[\beta ]}\leq c(\alpha ,d)(\lambda
^{-1}\wedge T)|\partial ^{k}f_{n}|_{\beta -[\beta ]}
\end{equation*}%
and there is a constant $C$ such that for all $s\leq t\leq T,$ 
\begin{equation*}
|\partial ^{k}u_{n}(t,\cdot )-u_{n}(s,\cdot )|_{\alpha /2+\beta }\leq
C(t-s)^{1/2}|\partial ^{k}f_{n}|_{\beta }
\end{equation*}%
for all $k\leq \lbrack \beta ]$. Let $[\beta ]<\beta ^{\prime }<\beta .$
Then, there is a constant $C$ depending only on $\alpha ,\ \beta ^{\prime
},T,d$, $\mu ,K$ such that 
\begin{equation}
|u_{n}|_{\alpha +\beta ^{\prime }-[\beta ]}\leqslant C|f_{n}|_{\beta
^{\prime }-[\beta ]}.  \label{du}
\end{equation}%
Moreover,%
\begin{equation}
|u_{n}|_{\beta -[\beta ]}\leq c(\alpha ,d)(\lambda ^{-1}\wedge
T)|f_{n}|_{\beta -[\beta ]}  \label{trys}
\end{equation}%
and there is a constant $C$ such that for all $s\leq t\leq T,$ 
\begin{equation}
|u_{n}(t,\cdot )-u_{n}(s,\cdot )|_{\alpha /2+\beta }\leq
C(t-s)^{1/2}|f_{n}|_{\beta }.  \label{ket}
\end{equation}%
By Lemma \ref{pagle1} and (\ref{nulis}), there exists $u\in C^{\alpha +\beta
^{\prime }}(H)$ such that $u_{n}\rightarrow u$ in $C^{\alpha +\beta ^{\prime
}}(H)$. Therefore $u$ satisfies (\ref{defs}) with $A_{\alpha }^{0}$ instead
of $\mathcal{L}^{(\alpha )}$. Since (\ref{vienas}) holds with $\beta
^{\prime }=\beta $, the solution $u\in C^{\alpha +\beta }(H)$ and the
statement is proved.\newline
\end{proof}

\subsubsection{Estimates of $\mathcal{B}^{(\protect\alpha )}f,f\in C^{%
\protect\alpha +\protect\beta }$}

We will use the following equality for the estimates of $\mathcal{B}%
^{(\alpha )}.$

\begin{lemma}
\label{r1}$($Lemma 2.1 in \textup{\cite{Kom84}}$)$ For $\delta \in (0,1)$
and $u\in C_{0}^{\infty }(\mathbf{R}^{d})$, 
\begin{equation}
u\big( x+y\big) -u(x)=K\int k^{(\delta )}(y,z)\partial ^{\delta }u(x-z)dz,
\label{eqn:diff_bound}
\end{equation}%
where $K=K(\delta ,d)$ is a constant, 
\begin{equation*}
k^{(\delta )}(y,z)=|z+y|^{-d+\delta }-|z|^{-d+\delta },
\end{equation*}%
and there exists a constant $C$ such that 
\begin{equation*}
\int |k^{(\delta )}(y,z)|dz\leq C|y|^{\delta },\forall y\in \mathbf{R}^{d}.
\end{equation*}
\end{lemma}

First we prove the following auxiliary estimate.

\begin{lemma}
\label{nle1} Let $\lambda \geq 1,\eta (d\upsilon )$ be a nonnegative measure
on $U$ and let $\mu \in \mathbf{N}_{0}^{d}$ be a mutiindex such that $n\geq
|\mu |,$ and $\sum_{j=1}^{k}\gamma _{j}=\mu $ with $\gamma _{j}\in \mathbf{N}%
_{0}^{d},\gamma _{j}\neq 0,k\geq \lambda $. Then there exist numbers $\theta
(\lambda ,j)\in \lbrack 0,1]$ and a constant $C$ such that for any
nonnegative measurable functions $l_{\lambda ,\gamma _{j}}$ on $U$,%
\begin{eqnarray*}
\int_{U}\prod_{j}|l_{\lambda ,\gamma _{j}}|d\eta &\leq &C\prod_{j}\Big(%
\int_{U_{1}}|l_{\lambda ,{\gamma_j}}|^{\frac{n}{|\gamma _{j}|}\vee \lambda
}d\eta \Big)^{( \frac{|\gamma _{j}|}{n}\wedge \lambda ) \theta (\lambda ,j)}%
\Big(\int_{U}|l_{\lambda ,{\gamma_j}}|^{\lambda }d\eta \Big)^{\frac{1}{%
\lambda }(1-\theta (\alpha ,j))} \\
&\leq &C\prod_{j}\Big\{ \Big(\int_{U_{1}}|l_{\lambda ,{\gamma_j}}|^{\frac{n}{%
|\gamma _{j}|}\vee \lambda }d\eta \Big)^{( \frac{|\gamma _{j}|}{n}\wedge
\lambda ) }+ \Big(\int_{U}|l_{\lambda ,{\gamma_j}}|^{\lambda }d\eta \Big)^{%
\frac{1}{\lambda }}\Big\} .
\end{eqnarray*}%
In addition, there is a constant $C$ such that%
\begin{equation*}
\int_{U}\prod_{j}|l_{\lambda ,\gamma _{j}}|d\eta \leq C\sum_{j}\int_{U}\big[%
|l_{\lambda ,{\gamma_j}}|^{\frac{n}{|\gamma _{j}|}\vee \lambda }+|l_{\lambda
,{\gamma_j}}|^{\lambda }\big]d\eta .
\end{equation*}
\end{lemma}

\begin{proof}
If there is $\gamma _{j_{0}}$ for which $\frac{|\mu |}{|\gamma _{j_{0}}|}%
<\lambda $ or $|\gamma _{j_{0}}|>\frac{|\mu |}{\lambda }$ ($\lambda >1\,$\
in this case and there could be only one $\gamma _{j_{0}}$ like this), then $%
|\mu |-|\gamma _{j_{0}}|<|\mu |(1-\frac{1}{\lambda })$ and for $\gamma
_{j}\neq \gamma _{j_{0}}$%
\begin{equation*}
\lambda \leq \frac{\lambda }{\lambda -1}\frac{|\mu |-|\gamma _{j_{0}}|}{%
|\gamma _{j}|}\leq \frac{|\mu |}{|\gamma _{j}|}\leq \frac{n}{|\gamma _{j}|}.
\end{equation*}%
By H\"{o}lder's inequality,%
\begin{eqnarray*}
\int_{U}\prod_{j}|l_{\lambda ,\gamma _{j}}|d\eta &\leq &\Big(%
\int_{U}|l_{\lambda ,\gamma _{j_{0}}}|^{\lambda }d\eta \Big)^{\frac{1}{%
\lambda }} \Big(\int_{U}\prod_{j\neq j_{0}}|l_{\lambda ,\gamma _{j}}|^{\frac{%
\lambda }{\lambda -1}}d\eta \Big)^{1-\frac{1}{\lambda }} \\
&\leq &C \Big[\int_{U}|l_{\lambda ,\gamma _{j_{0}}}|^{\lambda }d\eta
+\int_{U}\prod_{j\neq j_{0}}|l_{\lambda ,\gamma _{j}}|^{\frac{\lambda }{%
\lambda -1}}d\eta \Big],
\end{eqnarray*}%
where $\sum_{j\neq j_{0}}\gamma _{j}=\mu -\gamma _{j_{0}}$ and $\sum_{j\neq
j_{0}}\frac{|\gamma _{j}|}{|\mu |-|\gamma _{j_{0}}|}=1$. Hence, by H\"{o}%
lder's inequality,%
\begin{eqnarray*}
\int_{U}\prod_{j\neq j_{0}}|l_{\lambda ,\gamma _{j}}|^{\frac{\lambda }{%
\lambda -1}}d\eta &\leq &\prod_{j\neq j_{0}}\Big(\int_{U}|l_{\lambda ,\gamma
_{j}}|^{\frac{\lambda }{\lambda -1}\frac{|\mu |-|\gamma _{j_{0}}|}{|\gamma
_{j}|}}d\eta \Big)^{\frac{|\gamma _{j}|}{|\mu |-|\gamma _{j_{0}}|}} \\
&\leq &C\sum_{j\neq j_0}\int_{U}|l_{\lambda ,\gamma _{j}}|^{\frac{\lambda }{%
\lambda -1}\frac{|\mu |-|\gamma _{j_{0}}|}{|\gamma _{j}|}}d\eta \\
&\leq &C\sum_{j\neq j_0}\int_{U}\big(|l_{\lambda ,\gamma _{j}}|^{\lambda
}+|l_{\lambda ,\gamma _{j}}|^{\frac{n}{|\gamma _{j}|}}\big)d\eta ,
\end{eqnarray*}%
and by the interpolation inequality there are $\theta (\lambda ,j)\in
\lbrack 0,1]$ such that%
\begin{eqnarray*}
\int_{U}\prod_{j}|l_{\lambda ,\gamma _{j}}|d\eta &\leq &\Big(%
\int_{U}|l_{\lambda ,\gamma _{j_{0}}}|^{\lambda }d\pi \Big)^{\frac{1}{%
\lambda }}\prod_{j\neq j_{0}}\Big(\int_{U}|l_{\lambda ,\gamma _{j}}|^{\frac{%
\lambda }{\lambda -1}\frac{|\mu |-|\gamma _{j_{0}}|}{|\gamma _{j}|}}d\pi %
\Big)^{\frac{\lambda -1}{\lambda }\frac{|\gamma _{j}|}{|\mu |-|\gamma
_{j_{0}}|}} \\
&\leq &\prod_{j}\Big(\int_{U}|l_{\lambda ,\gamma _{j}}|^{\lambda }d\pi \Big)%
^{\frac{1}{\lambda }(1-\theta (\lambda ,j))} \Big(\int_{U}|l_{\lambda
,\gamma _{j}}|^{\frac{n}{|\gamma _{j}|}\vee \lambda }d\pi \Big)^{(\frac{%
|\gamma _{j}|}{n}\wedge \frac{1}{\lambda })\theta (\lambda ,j)}.
\end{eqnarray*}

If for all $j,$ $\frac{|\mu |}{|\gamma _{j}|}\geq \lambda $, then $\sum_{j}%
\frac{|\gamma _{j}|}{|\mu |}=1$ and by H\"{o}lder's inequality, 
\begin{eqnarray*}
\int_{U}\prod_{j}|l_{\lambda ,\gamma _{j}}|d\eta &\leq &\prod_{j}\Big(%
\int_{U}|l_{\lambda ,\gamma _{j}}|^{\frac{|\mu |}{|\gamma _{j}|}}d\eta \Big)%
^{\frac{|\gamma _{j}|}{|\mu |}} \\
&\leq &C\sum_{j}\int_{U}|l_{\lambda ,\gamma _{j}}|^{\frac{|\mu |}{|\gamma
_{j}|}}d\eta \\
& \leq & C\sum_{j}\int_{U_{1}}\big[|l_{\lambda ,\gamma _{j}}|^{\frac{n}{%
|\gamma _{j}|}}+|l_{\lambda ,\gamma _{j}}|^{\lambda }\big]d\eta .
\end{eqnarray*}%
Also, by interpolation inequalities, 
\begin{eqnarray*}
\prod_{j}\Big(\int_{U}|l_{\lambda ,\gamma _{j}}|^{\frac{|\mu |}{|\gamma _{j}|%
}}d\eta \Big)^{\frac{|\gamma _{j}|}{|\mu |}} \leq \prod_{j}\Big(%
\int_{U}|l_{\lambda ,\gamma _{j}}|^{\lambda }d\eta \Big)^{\frac{1}{\lambda }%
(1-\theta (\lambda ,j)}\Big(\int_{U_{1}}|l_{\lambda ,\gamma _{j}}|^{\frac{n}{%
|\gamma _{j}|}}d\eta \Big)^{\frac{|\gamma _{j}|}{n}\theta (\lambda ,j)}.
\end{eqnarray*}%
The statement follows.
\end{proof}

\begin{proposition}
\label{b1}Let $\beta >0,\beta \notin \mathbf{N}$. Assume \textup{A1-A4($%
\beta )$} hold. Then for each $\varepsilon >0$ there exists a constant $%
C_{\varepsilon }$ such that%
\begin{equation*}
|\mathcal{B}^{(\alpha )}f|_{\beta }\leq \varepsilon |f|_{\alpha +\beta
}+C_{\varepsilon }|f|_{\beta -[\beta ]},f\in C^{\alpha +\beta }(\mathbf{R}%
^{d}).
\end{equation*}
\end{proposition}

\begin{proof}
For $\gamma \in \mathbf{N}_{0}^{d},|\gamma |\leq \lbrack \beta ],x\in 
\mathbf{R}^{d},$%
\begin{eqnarray*}
\partial ^{\gamma }[\mathcal{B}_{x}^{(\alpha )}f(x)] &=&\sum_{\nu +\mu
=\gamma }\partial _{z}^{\mu }\mathcal{B}_{z}^{(\alpha )}\partial ^{\nu
}f(x)|_{z=x} \\
&=&\mathcal{B}_{x}^{(\alpha )}\partial ^{\gamma }f(x)+\sum_{\nu +\mu =\gamma
,\mu \neq 0}\partial _{z}^{\mu }\mathcal{B}_{z}^{(\alpha )}\partial ^{\nu
}f(x)|_{z=x}.
\end{eqnarray*}%
For $\mu \neq 0$, 
\begin{eqnarray}
&&\partial _{z}^{\mu }\mathcal{B}_{z}^{(\alpha )}\partial ^{\nu }f(x)|_{z=x}
\notag  \label{nf1} \\
&=&\mathbf{1}_{\{\alpha >1\}}\int_{U_{1}}\partial _{z}^{\mu }\big[\partial
^{\nu }f(x+l_{\alpha }(z,\upsilon ))-\partial ^{\nu }f(x)-\big(\nabla
f(x),l_{\alpha }(z,\upsilon )\big)\big]|_{z=x}\pi (d\upsilon )  \notag \\
&&+\int \theta _{\alpha }(\upsilon )\partial _{z}^{\mu }\partial ^{\nu
}f(x+l_{\alpha }(z,\upsilon ))|_{z=x}\pi (d\upsilon )=T_{1}(x)+T_{2}(x),
\end{eqnarray}%
with $\theta _{\alpha }(\upsilon )=\mathbf{1}_{\{\alpha \leq 1\}}+\mathbf{1}%
_{\{\alpha >1\}}\mathbf{1}_{U_{1}^{c}}(\upsilon )$, and%
\begin{eqnarray*}
&&\mathcal{B}_{x}^{(\alpha )}\partial ^{\gamma }f(x) \\
&=&\int \theta _{\alpha }(\upsilon )\big[\partial ^{\gamma }f(x+l_{\alpha
}(x,\upsilon ))-\partial ^{\gamma }f(x)\big]d\pi \\
&&+\mathbf{1}_{\{\alpha >1\}}\int_{U_{1}}\big[\partial ^{\nu }f(x+l_{\alpha
}(x,\upsilon ))-\partial ^{\nu }f(x)-(\nabla f(x),l_{\alpha }(x,\upsilon ))%
\big]\pi (d\upsilon ) \\
&=&S_{1}(x)+S_{2}(x).
\end{eqnarray*}

\emph{Estimates of }$S_{1}$. For any $\beta ^{\prime }\in ([\beta ],\beta )$
there is a constant $C$ such that%
\begin{eqnarray*}
\Big|\int \theta _{\alpha }(\upsilon )[\partial ^{\gamma }f(x+l_{\alpha
}(x,\upsilon ))-\partial ^{\gamma }f(x)]d\pi \Big| \leq C|f|_{\beta ^{\prime
}}\int \theta _{\alpha }(\upsilon )|l_{\alpha }(x,\upsilon )|^{\alpha \wedge
1}\wedge 1d\pi .
\end{eqnarray*}

For $x,x^{\prime }\in \mathbf{R}^{d},$%
\begin{equation*}
|S_{1}(x)-S_{1}(x^{\prime })|\leq S_{11}+S_{12},
\end{equation*}%
where%
\begin{eqnarray*}
S_{11} &=&\int \theta _{\alpha }(\upsilon )\big|\lbrack \partial ^{\gamma
}f(x+l_{\alpha }(x,\upsilon ))-\partial ^{\gamma }f(x)] \\
&&-[\partial ^{\gamma }f(x^{\prime }+l_{\alpha }(x,\upsilon ))-\partial
^{\gamma }f(x^{\prime })]\big|d\pi, \\
S_{12} &=&\int \theta _{\alpha }(\upsilon )\big|\partial ^{\gamma
}f(x^{\prime }+l_{\alpha }(x,\upsilon ))-\partial ^{\gamma }f(x^{\prime
}+l_{\alpha }(x^{\prime },\upsilon ))\big|d\pi .
\end{eqnarray*}%
For $\beta ^{\prime }<\beta $, by assumption \textup{A3}($\beta )$,%
\begin{eqnarray*}
S_{12} &\leq &C|f|_{\beta ^{\prime }}\int \theta _{\alpha }(\upsilon
)|\Delta l_{\alpha }(x,x^{\prime },\upsilon )|^{(\alpha +\beta ^{\prime
}-[\beta ])\wedge 1}\wedge 1d\pi \\
&\leq &C|x-x^{\prime }|^{\beta -[\beta ]},
\end{eqnarray*}%
with $\Delta l_{\alpha }(x,x^{\prime },\upsilon )=l_{\alpha }(x,\upsilon
)-l_{\alpha }(x^{\prime },\upsilon )$. For each $n$, by Lemma \ref{r1},%
\begin{eqnarray*}
S_{11} &=&\int_{U_{n}}...+\int_{U_{n}^{c}}... \\
&\le & \mathbf{1}_{\{\alpha <1\}}\int_{U_{n}}\big||\partial |^{\alpha
}\partial ^{\gamma }f(x-z)-|\partial |^{\alpha }\partial ^{\gamma }f(x-z)%
\big| k^{(\alpha )}(l_{\alpha }(x,\upsilon ),z)dzd\pi \\
&&+|f|^{\beta }|x-x^{\prime }|^{\beta -[\beta ]} \\
&\leq &C[|f|_{\alpha +\beta }\mathbf{1}_{\{\alpha
<1\}}\sup_{x}\int_{U_{n}}l_{\alpha }(x,\upsilon )^{\alpha }d\pi +|f|^{\beta
}]|x-x^{\prime }|^{\beta -[\beta ]}
\end{eqnarray*}
\end{proof}

\begin{proof}
\emph{Estimates of }$S_{2}$. Let $\alpha >1$. For $g\in C^{\alpha +\beta
-[\beta ]}$, denote 
\begin{eqnarray*}
\mathcal{T}_{h}g(x) &=&g(x+h)-g(x)-(\nabla g(x),h), \\
D_{h}\nabla g(x) &=&\nabla g(x+h)-\nabla g(x),x,h\in \mathbf{R}^{d}.
\end{eqnarray*}%
By lemma \ref{r1},%
\begin{eqnarray}
\mathcal{T}_{h}g(x) &=&\int_{0}^{1}\big(\nabla g(x+sh\big)-\nabla g(x),h\big)%
ds  \label{nf5} \\
&=&\int_{0}^{1}\int \big(|\partial |^{\alpha -1}\nabla g(x-z)k^{(\alpha
-1)}(sh,z),h\big)dzds,x,h\in \mathbf{R}^{d}.  \notag
\end{eqnarray}%
For $\alpha >1,$%
\begin{eqnarray*}
S_{2} &=&\int_{U_{1}}\mathcal{T}_{l_{\alpha }(x,v)}\partial ^{\nu }f(x)\pi
(d\upsilon ) \\
&=&\int_{U_{1}}\int_{0}^{1}\big(D_{sl_{\alpha }(x,\upsilon )}\nabla \partial
^{\nu }f(x),l_{\alpha }(x,\upsilon )\big)ds\pi (d\upsilon )
\end{eqnarray*}%
and for any $\beta ^{\prime }\in ([\beta ],\beta )$, 
\begin{equation*}
|S_{2}(x)|\leq C|f|_{a+\beta ^{\prime }}\int |l_{\alpha }(x,\upsilon
)|^{\alpha }d\pi \leq C|f|_{a+\beta ^{\prime }}.
\end{equation*}%
For $x,x^{\prime }\in \mathbf{R}^{d},\alpha >1,$%
\begin{eqnarray*}
S_{2}(x)-S_{2}(x^{\prime }) &=&\int_{U_{1}}[\mathcal{T}_{l_{\alpha
}(x,v)}\partial ^{\nu }f(x)-\mathcal{T}_{l_{\alpha }(x,v)}\partial ^{\nu
}f(x^{\prime })]d\pi \\
&&+\int_{U_{1}}[\mathcal{T}_{l_{\alpha }(x,v)}\partial ^{\nu }f(x^{\prime })-%
\mathcal{T}_{l_{\alpha }(x^{\prime },v)}\partial ^{\nu }f(x^{\prime })]d\pi
\\
&=&S_{21}+S_{22}.
\end{eqnarray*}%
Since for any $\beta ^{\prime }\in ([\beta ],\beta ),$ 
\begin{eqnarray*}
&&|T_{l_{\alpha }(x,v)}\partial ^{\nu }f(x^{\prime })-T_{l_{\alpha
}(x^{\prime },v)}\partial ^{\nu }f(x^{\prime })| \\
&\leq &C|f|_{\alpha +\beta ^{\prime }}\big(|l_{\alpha }(x^{\prime },\upsilon
)|^{\alpha -1}+|l_{\alpha }(x,\upsilon )|^{\alpha -1}\big)|\Delta l_{\alpha
}(x,x^{\prime },\upsilon )|,
\end{eqnarray*}%
then by H\"{o}lder's inequality, 
\begin{eqnarray*}
|S_{22}| &\leq &C|f|_{\alpha +\beta ^{\prime }}\Big(\int_{U_{1}}|\Delta
l_{\alpha }(x,x^{\prime },\upsilon )|^{\alpha }d\pi \Big)^{1/\alpha } \\
&\leq &C|f|_{\alpha +\beta ^{\prime }}|\beta -\beta ^{\prime }|.
\end{eqnarray*}%
By Lemma \ref{r1} and (\ref{nf5}), for each $n$ and $\beta ^{\prime }\in
([\beta ],\beta )$, there is a constant $C^{\prime }$ such that%
\begin{eqnarray*}
|S_{21}| &\leq &\int_{U_{n}}\int_{0}^{1}\int |\partial ^{\alpha -1}\nabla
\partial ^{\nu }f(x-z)-\partial ^{\alpha -1}\nabla \partial ^{\nu
}f(x^{\prime }-z)| \\
&&\times |k^{(\alpha -1)}(sl_{\alpha }(x,\upsilon ),z)|~|l_{\alpha
}(x,\upsilon )|dzds\pi (d\upsilon ) \\
&&+\int_{U_{1}\backslash U_{n}}|\mathcal{T}_{l_{\alpha }(x,v)}\partial ^{\nu
}f(x)-\mathcal{T}_{l_{\alpha }(x,v)}\partial ^{\nu }f(x^{\prime })|d\pi \\
&\leq &C|f|_{\alpha +\beta }|x-x^{\prime }|^{\beta -[\beta
]}\int_{U_{n}}|l_{\alpha }(x,\upsilon )|^{\alpha }d\pi \\
&&+C^{\prime }|f|_{\alpha +\beta ^{\prime }}|x-x^{\prime }|^{\beta -[\beta
]}\int_{U_{1}\backslash U_{n}}(1+|l_{\alpha }(x,\upsilon )|)d\pi .
\end{eqnarray*}

\emph{Estimates of }$T_{1}$. If $\alpha >1,$ then $%
T_{1}(x)=A_{1}(x)+A_{2}(x),$ where 
\begin{equation*}
A_{1}(x)=\int_{U_{1}}\big(\nabla \partial ^{\nu }f(x+l_{\alpha }(z,\upsilon
))-\nabla \partial ^{\nu }f(x),\partial _{z}^{\mu }l_{\alpha }(z,\upsilon )%
\big)|_{z=x}d\pi ,x\in \mathbf{R}^{d},
\end{equation*}%
and $A_{2}(x)$ consists of the sum whose terms are of the form 
\begin{equation*}
\int_{U_{1}}\partial ^{\nu +\kappa }f(x+l_{\alpha })\prod_{\kappa _{i}\neq
0,j}\partial ^{\gamma _{j}^{i}}l_{\alpha }^{i}d\pi 
\end{equation*}%
with the non-zero multiindices $\gamma _{j}^{i}\in \mathbf{N}_{0}^{d}$ such
that $\sum_{\kappa _{i}\neq 0,j}\gamma _{j}^{i}=\mu $ and $|\mu |\geq
|\kappa |\geq 2$.

Applying H\"{o}lder's inequality, we have 
\begin{eqnarray}
|A_{1}(x)| &\leq &C|f|_{\alpha +[\beta ]}\int_{U_{1}}(|l_{\alpha
}(x,\upsilon )|\wedge 1)^{\alpha -1}|\partial _{z}^{\mu }l_{\alpha
}(x,\upsilon )|d\pi   \notag  \label{nf2} \\
&\leq &C|f|_{\alpha +[\beta ]}\Big(\int_{U_{1}}(|l_{\alpha }(x,\upsilon
)|\wedge 1)^{\alpha }d\pi \Big)^{1-\frac{1}{\alpha }}\Big(%
\int_{U_{1}}|\partial _{z}^{\mu }l_{\alpha }(x,\upsilon )|^{\alpha }d\pi %
\Big)^{1/\alpha }  \notag \\
&\leq &C|f|_{\alpha +[\beta ]}.
\end{eqnarray}%
Obviously,%
\begin{equation*}
\Big|\int_{U_{1}}\partial ^{\nu +\kappa }f(x+l_{\alpha }(x,\upsilon
))\prod_{\kappa _{i}\neq 0,j}\partial ^{\gamma _{j}^{i}}l_{\alpha
}(x,\upsilon )^{i}d\pi \Big|\leq |f|_{\beta }\int_{U_{1}}\prod_{\kappa
_{i}\neq 0,j}|\partial ^{\gamma _{j}^{i}}l_{\alpha }^{i}(x,\upsilon )|d\pi .
\end{equation*}%
By Lemma \ref{nle1}, 
\begin{equation*}
\int_{U_{1}}\prod_{\substack{ \kappa _{i}\neq 0, \\ (i,j)\neq (i_{0},j_{0})}}%
|\partial ^{\gamma _{j}^{i}}l_{\alpha }^{i}|^{\frac{\alpha }{\alpha -1}}d\pi
\leq C\prod_{\kappa _{i}\neq 0,j}\Big[\Big(\int_{U_{1}}|\partial ^{\gamma
_{j}^{i}}l_{\alpha }^{i}|^{\alpha }d\pi \Big)^{\frac{1}{\alpha }}+\Big(\int
|\partial ^{\gamma _{j}^{i}}l_{\alpha }^{i}|^{\frac{[\beta ]}{|\gamma
_{j}^{i}|}}d\pi \Big)^{\frac{|\gamma _{j}^{i}|}{[\beta ]}}\Big].
\end{equation*}%
Hence, $|A_{2}(x)|\leq C|f|_{\beta },x\in \mathbf{R}^{d}.$Thus there exists $%
\beta ^{\prime }<\beta $ such that $|T_{1}(x)|\leq C|f|_{\beta ^{\prime
}},x\in \mathbf{R}^{d}.$

Now we estimate the differences. For $x,x^{\prime }\in \mathbf{R}^{d}$ and a
multiindex $\sigma $, denote $\Delta \partial ^{\sigma }l_{\alpha
}(x,x^{\prime };\upsilon )=\partial ^{\sigma }l_{\alpha }(x,\upsilon
)-\partial ^{\sigma }l_{\alpha }(x^{\prime },\upsilon ),\upsilon \in U$. For
any $\beta ^{\prime }>[\beta ]+1,$ 
\begin{eqnarray*}
|A_{1}(x)-A_{1}(x^{\prime })| &\leq &|f|_{\beta ^{\prime }}\Big[%
\int_{U_{1}}(|\Delta l_{\alpha }(x,x^{\prime };\upsilon )|\wedge 1)|\partial
_{z}^{\mu }l_{\alpha }(x,\upsilon )|d\pi  \\
&&+\int_{U_{1}}(|l_{\alpha }(x^{\prime },\upsilon )|\wedge 1|)|\Delta
\partial _{z}^{\mu }l_{\alpha }(x,x^{\prime };\upsilon )|d\pi \Big] \\
&=&|f|_{\beta ^{\prime }}[A_{11}+A_{12}].
\end{eqnarray*}%
Now%
\begin{eqnarray*}
A_{12} &\leq &\Big(\int_{U_{1}}(|l_{\alpha }(x^{\prime },\upsilon )|\wedge
1)^{\frac{\alpha }{\alpha -1}}d\pi \Big)^{1-\frac{1}{\alpha }}\Big(%
\int_{U_{1}}|\Delta \partial _{z}^{\mu }l_{\alpha }(x,x^{\prime };\upsilon
)|^{\alpha }d\pi \Big)^{\frac{1}{\alpha }} \\
&\leq &\Big(\int_{U_{1}}(|l_{\alpha }(x^{\prime },\upsilon )|\wedge
1)^{\alpha }d\pi \Big)^{1-\frac{1}{\alpha }}\Big(\int_{U_{1}}|\Delta
\partial _{z}^{\mu }l_{\alpha }(x,x^{\prime };\upsilon )|^{\alpha }d\pi \Big)%
^{\frac{1}{\alpha }} \\
&\leq &C|x-x^{\prime }|^{\beta -[\beta ]}.
\end{eqnarray*}%
Obviously, 
\begin{eqnarray*}
A_{11} &=&\int_{U_{1}}(|\Delta l_{\alpha }(x,x^{\prime };\upsilon )|\wedge
1)|\partial ^{\mu }l_{\alpha }(x,\upsilon )|d\pi  \\
&\leq &\int_{U_{1}}\mathbf{1}_{\big\{|\partial _{z}^{\mu }l_{\alpha
}(x,\upsilon )|\leq 1\big\}}(|\Delta l_{\alpha }(x,x^{\prime };\upsilon
)|\wedge 1)|\partial ^{\mu }l_{\alpha }(x,\upsilon )|d\pi  \\
&&+\int_{U_{1}}\mathbf{1}_{\big\{|\partial _{z}^{\mu }l_{\alpha }(x,\upsilon
)|>1\big\}}(|\Delta l_{\alpha }(x,x^{\prime };\upsilon )|\wedge 1)|\partial
_{z}^{\mu }l_{\alpha }(x,\upsilon )|d\pi .
\end{eqnarray*}%
By H\"{o}lder's inequality,%
\begin{eqnarray*}
&&\int_{|\partial _{z}^{\mu }l_{\alpha }(x,\upsilon )|\leq 1}\mathbf{1}%
_{U_{1}}(\upsilon )(|\Delta l_{\alpha }(x,x^{\prime };\upsilon )|\wedge
1)|\partial ^{\mu }l_{\alpha }(x,\upsilon )|d\pi  \\
&\leq &\Big(\int_{|\partial ^{\mu }l_{\alpha }(x,\upsilon )|\leq 1}\mathbf{1}%
_{U_{1}}(\upsilon )(|\Delta l_{\alpha }(x,x^{\prime };\upsilon )|\wedge
1)^{\alpha }d\pi \Big)^{\frac{1}{\alpha }} \\
&&\times \Big(\int_{|\partial ^{\mu }l_{\alpha }(x,\upsilon )|\leq 1}\mathbf{%
1}_{U_{1}}(\upsilon )|\partial _{z}^{\mu }l_{\alpha }(x,\upsilon )|^{\frac{%
\alpha }{\alpha -1}}d\pi \Big)^{1-\frac{1}{\alpha }} \\
&\leq &C|x-x^{\prime }|^{\beta -[\beta ]}\big(\int_{|\partial _{z}^{\mu
}l_{\alpha }(x,\upsilon )|\leq 1}\mathbf{1}_{U_{1}}(\upsilon )|\partial
^{\mu }l_{\alpha }(x,\upsilon )|^{\alpha }d\pi \big)^{1-\frac{1}{\alpha }} \\
&\leq &C|x-x^{\prime }|^{\beta -[\beta ]},
\end{eqnarray*}%
and%
\begin{eqnarray*}
&&\int_{|\partial ^{\mu }l_{\alpha }(x,\upsilon )|>1}\mathbf{1}%
_{U_{1}}(\upsilon )(|\Delta l_{\alpha }(x,x^{\prime };\upsilon )|\wedge
1)|\partial ^{\mu }l_{\alpha }(x,\upsilon )|d\pi  \\
&\leq &\int_{|\partial ^{\mu }l_{\alpha }(x,\upsilon )|>1}\mathbf{1}%
_{U_{1}}(\upsilon )(|\Delta l_{\alpha }(x,x^{\prime };\upsilon )|\wedge
1)|\partial ^{\mu }l_{\alpha }(x,\upsilon )|^{\alpha }d\pi  \\
&\leq &C|x-x^{\prime }|^{\beta -[\beta ]}.
\end{eqnarray*}%
Hence,%
\begin{equation*}
A_{11}\leq C|x-x^{\prime }|^{\beta -[\beta ]},A_{12}\leq C|x-x^{\prime
}|^{\beta -[\beta ]}.
\end{equation*}

Since $A_{2}(x)$ consists of the sum whose terms are of the form 
\begin{equation*}
\int \partial ^{\nu +\kappa }f(x+l_{\alpha })\prod_{\kappa _{i}\neq
0,j}\partial ^{\gamma _{j}^{i}}l_{\alpha }^{i}d\pi
\end{equation*}%
with the non-zero multiindices $\gamma _{j}^{i}\in \mathbf{N}_{0}^{d}$ such
that $\sum_{\kappa _{i}\neq 0,j}\gamma _{j}^{i}=\mu $ and $|\mu |\geq
|\kappa |\geq 2,$ we estimate the differences of a generic term%
\begin{equation*}
\tilde{A}_{2}(x)=\int \partial ^{\nu +\kappa }f(x+l_{\alpha }(x,\upsilon
))\prod_{\kappa _{i}\neq 0,j}\partial ^{\gamma _{j}^{i}}l_{\alpha
}^{i}(x,\upsilon )d\pi .
\end{equation*}%
We have%
\begin{eqnarray*}
&&|\tilde{A}_{2}(x)-\tilde{A}_{2}(x^{\prime })| \\
&\leq &\int \big|\partial ^{\nu +\kappa }f(x+l_{\alpha }(x,\upsilon
))-\partial ^{\nu +\kappa }f(x^{\prime }+l_{\alpha }(x^{\prime },\upsilon ))%
\big|\prod_{\kappa _{i}\neq 0,j}|\partial ^{\gamma _{j}^{i}}l_{\alpha
}^{i}(x,\upsilon )|d\pi \\
&&+\sup_{x}|\partial ^{\nu +\kappa }f(x)|\int_{U_{1}}\big|\prod_{\kappa
_{i}\neq 0,j}\partial ^{\gamma _{j}^{i}}l_{\alpha }^{i}(x,\upsilon
)-\prod_{\kappa _{i}\neq 0,j}\partial ^{\gamma _{j}^{i}}l_{\alpha
}^{i}(x^{\prime },\upsilon )\big|d\pi \\
&=&\tilde{A}_{21}+\tilde{A}_{22}.
\end{eqnarray*}%
First, by Lemma \ref{nle1} with $\lambda =\alpha $,%
\begin{eqnarray*}
\tilde{A}_{21} &\leq &\int_{U_{1}}\big( |x-x^{\prime }|^{\beta -[\beta
]}+|\Delta l_{\alpha }(x,x^{\prime };\upsilon )|\wedge 1\big) \prod_{\kappa
_{i}\neq 0,j}|\partial ^{\gamma _{j}^{i}}l_{\alpha }^{i}(x,\upsilon )|d\pi \\
&\leq &C|f|_{\beta +1}\sum_{j}\int_{U_{1}}\big( |x-x^{\prime }|^{\beta
-[\beta ]}+|\Delta l_{\alpha }(x,x^{\prime };\upsilon )|\wedge 1\big) %
\lbrack |\partial ^{\gamma _{j}^{i}}l_{\alpha }|^{\frac{[\beta ]}{|\gamma
_{j}|}\vee \alpha }+|\partial ^{\gamma _{j}^{i}}l_{\alpha }|^{\alpha }]d\pi
\\
&\leq &C|f|_{\beta +1}|x-x^{\prime }|^{\beta -[\beta ]},
\end{eqnarray*}%
and%
\begin{eqnarray*}
\tilde{A}_{22} &\leq &|f|_{\beta }\int_{U_{1}}\big|\prod_{\kappa _{i}\neq
0,j}\partial ^{\gamma _{j}^{i}}l_{\alpha }^{i}(x,\upsilon )-\prod_{\kappa
_{i}\neq 0,j}\partial ^{\gamma _{j}^{i}}l_{\alpha }^{i}(x^{\prime },\upsilon
)\big|d\pi \\
&\leq &C|f|_{\beta }\sum_{\kappa _{i}\neq 0,j}\Big(\int_{U_{1}}|\Delta
\partial ^{\gamma _{j}^{i}}l_{\alpha }^{i}(x,x^{\prime },\upsilon )|^{\frac{%
[\beta ]}{|\gamma _{j}|}\vee \alpha }d\eta \Big)^{\big( \frac{|\gamma _{j}|}{%
[\beta ]}\wedge \alpha \big) \theta (\alpha ,j)} \\
&&\times \Big(\int_{U_{1}}|\Delta \partial ^{\gamma _{j}^{i}}l_{\alpha
}^{i}(x,x^{\prime },\upsilon )|^{\alpha }d\eta \Big)^{\alpha (1-\theta
(\alpha ,j))} \\
&\leq &C|f|_{\beta }|x-x^{\prime }|^{\beta -[\beta ]}.
\end{eqnarray*}

\emph{Estimate of }$T_{2}$. \ The part $T_{2}(x)$ consists of the sum whose
terms are of the form 
\begin{equation*}
\int \theta _{\alpha }(\upsilon )\partial ^{\nu +\kappa }f(x+l_{\alpha
})\prod_{\kappa _{i}\neq 0,j}\partial ^{\gamma _{j}^{i}}l_{\alpha }^{i}d\pi
\end{equation*}%
with the non-zero multiindices $\gamma _{j}^{i}\in \mathbf{N}_{0}^{d}$ such
that $\sum_{\kappa _{i}\neq 0,j}\gamma _{j}^{i}=\mu $ and $|\mu |\geq
|\kappa |\geq 1$. By Lemma \ref{nle1},%
\begin{eqnarray*}
&&\big\vert \int \theta _{\alpha }(\upsilon )\partial ^{\nu +\kappa
}f(x+l_{\alpha })\prod_{\kappa _{i}\neq 0,j}\partial ^{\gamma
_{j}^{i}}l_{\alpha }^{i}d\pi \big\vert \\
&\leq &\sup_{x}|\partial ^{\nu +\kappa }f(x)|\prod_{\kappa _{i}\neq 0,j} %
\Big[ \Big( \int \theta _{\alpha }(\upsilon )|\partial ^{\gamma
_{j}^{i}}l_{\alpha }^{i}|d\pi \Big) +\Big( \int \theta _{\alpha }(\upsilon
)|\partial ^{\gamma _{j}^{i}}l_{\alpha }^{i}|^{\frac{[\beta ]}{|\gamma
_{j}^{i}|}}d\pi \Big) ^{\frac{|\gamma _{j}^{i}|}{[\beta ]}}\Big] \\
&\leq &C|f|_{\beta }.
\end{eqnarray*}%
For $x,x^{\prime }\in \mathbf{R}^{d}$, 
\begin{eqnarray*}
&&\Big|\int \theta _{\alpha }(\upsilon )\partial ^{\nu +\kappa
}f(x+l_{\alpha }(x,\upsilon ))\prod_{\kappa _{i}\neq 0,j}\partial ^{\gamma
_{j}^{i}}l_{\alpha }^{i}(x,\upsilon )d\pi \\
&&-\int \theta _{\alpha }(\upsilon )\partial ^{\nu +\kappa }f(x^{\prime
}+l_{\alpha }(x^{\prime },\upsilon ))\prod_{\kappa _{i}\neq 0,j}\partial
^{\gamma _{j}^{i}}l_{\alpha }^{i}(x^{\prime },\upsilon )d\pi \Big| \\
&\leq &C\int \big|\partial ^{\nu +\kappa }f(x+l_{\alpha }(x,\upsilon
))-\partial ^{\nu +\kappa }f(x^{\prime }+l_{\alpha }(x^{\prime },\upsilon ))%
\big|\prod_{\kappa _{i}\neq 0,j}|\partial ^{\gamma _{j}^{i}}l_{\alpha
}^{i}(x,\upsilon )|d\pi \\
&&+|f|_{\beta }\int \big|\prod_{\kappa _{i}\neq 0,j}\partial ^{\gamma
_{j}^{i}}l_{\alpha }^{i}(x,\upsilon )-\prod_{\kappa _{i}\neq 0,j}\partial
^{\gamma _{j}^{i}}l_{\alpha }^{i}(x^{\prime },\upsilon )\big|d\pi \\
&=&A+B.
\end{eqnarray*}

For the first term, by assumption \textup{A3}($\beta )$ with $\beta ^{\prime
}<\beta $, 
\begin{eqnarray*}
A &\leq &C|x-x^{\prime }|^{\beta -[\beta ]}|f|_{\beta }\int \theta _{\alpha
}(\upsilon )\prod_{\kappa _{i}\neq 0,j}|\partial ^{\gamma _{j}^{i}}l_{\alpha
}^{i}(x,\upsilon )|d\pi  \\
&&+|f|_{\alpha +\beta ^{\prime }}\int \theta _{\alpha }(\upsilon )\big(%
|\Delta l_{\alpha }(x,x^{\prime };\upsilon )|\wedge 1\big)^{(\alpha +\beta
^{\prime }-[\beta ])\wedge 1}\prod_{\kappa _{i}\neq 0,j}|\partial ^{\gamma
_{j}^{i}}l_{\alpha }^{i}(x,\upsilon )|d\pi \}
\end{eqnarray*}%
By Lemma \ref{nle1} with $\lambda =1$,%
\begin{eqnarray*}
&&\int \theta _{\alpha }(\upsilon )\big(|\Delta c(x,x^{\prime };\upsilon
)|\wedge 1\big)^{(\alpha +\beta ^{\prime }-[\beta ])\wedge 1}\prod_{\kappa
_{i}\neq 0,j}|\partial ^{\gamma _{j}^{i}}c^{i}(x,\upsilon )|d\pi  \\
&\leq &\sum_{\kappa _{i}\neq 0,j}\int \theta _{\alpha }(\upsilon )\big(%
|\Delta c(x,x^{\prime };\upsilon )|\wedge 1\big)^{(\alpha +\beta ^{\prime
}-[\beta ])\wedge 1}(|\partial ^{\gamma _{j}^{i}}c^{i}|+|\partial ^{\gamma
_{j}^{i}}c^{i}|^{\frac{[\beta ]}{|\gamma _{j}^{i}|}})d\pi  \\
&\leq &C|x-x^{\prime }|^{\beta -[\beta ]}.
\end{eqnarray*}%
and%
\begin{equation*}
\int \theta _{\alpha }(\upsilon )\prod_{\kappa _{i}\neq 0,j}|\partial
^{\gamma _{j}^{i}}l_{\alpha }^{i}(x,\upsilon )|d\pi \leq \sum_{\kappa
_{i}\neq 0,j}\int \theta _{\alpha }(\upsilon )(|\partial ^{\gamma
_{j}^{i}}l_{\alpha }^{i}|+|\partial ^{\gamma _{j}^{i}}l_{\alpha }^{i}|^{%
\frac{[\beta ]}{|\gamma _{j}^{i}|}}d\pi \leq C.
\end{equation*}%
Hence, $A\leq C|f|_{\alpha +\beta ^{\prime }}|x-x^{\prime }|^{\beta -[\beta
]}.$

By lemma \ref{nle1} with $\lambda =1,$%
\begin{eqnarray*}
&&\int \theta _{\alpha }(\upsilon )\big|\prod_{\kappa _{i}\neq 0,j}\partial
^{\gamma _{j}^{i}}l_{\alpha }^{i}(x,\upsilon )-\prod_{\kappa _{i}\neq
0,j}\partial ^{\gamma _{j}^{i}}l_{\alpha }^{i}(x^{\prime },\upsilon )\big|%
d\pi \\
&\leq &C\sum_{\kappa _{i}\neq 0,j} \Big\{\Big(\int \theta _{\alpha
}(\upsilon )\big|\partial ^{\gamma _{j}^{i}}l_{\alpha }^{i}(x,\upsilon
)-\partial ^{\gamma _{j}^{i}}l_{\alpha }^{i}(x^{\prime },\upsilon )\big|^{%
\frac{[\beta ]}{|\gamma _{j}^{i}|}}d\pi \Big)^{\frac{|\gamma _{j}^{i}|}{%
[\beta ]}} \\
&&+\int \theta _{\alpha }(\upsilon )|\partial ^{\gamma _{j}^{i}}l_{\alpha
}^{i}(x,\upsilon )-\partial ^{\gamma _{j}^{i}}l_{\alpha }^{i}(x^{\prime
},\upsilon )|d\pi \Big\} \\
&\leq &C|x-x^{\prime }|^{\beta -[\beta ]}.
\end{eqnarray*}%
Thus%
\begin{equation*}
|T_{2}(x)-T_{2}(x^{\prime })|\leq C|f|_{\alpha +\beta ^{\prime
}}|x-x^{\prime }|^{\beta -[\beta ]}.
\end{equation*}

The statement follows by the standard interpolation inequalities.
\end{proof}

\subsection{Proof of Theorem \protect\ref{thm:StoCP} and Corollary \protect
\ref{lcornew1}}


It is well known that for an arbitrary but fixed $\delta >0$ there is a
family of cubes $D_{k}\subseteq \tilde{D}_{k}\subseteq \mathbf{R}^{d}$ and a
family of $\eta _{k}\in C_{0}^{\infty }(\mathbf{R}^{d})$ with the following
properties:

\begin{enumerate}
\item For all $k\geq 1,D_{k}$ and $\tilde{D}_{k}$ have a common center $%
x_{k},$ diam$D_{k}\leq \delta$, dist$(D_{k},\mathbf{R}^{d}\backslash \tilde{D%
}_{k})\leq N\delta $ for a certain constant $N=N(d)>0$, $\bigcup _{k}D_{k}=%
\mathbf{R}^{d}$, and $1\leq \sum_{k}\mathbf{1}_{\tilde{D}_{k}}\leq 2^{d}.$

\item For all $k$, $0\leq \eta _{k}\leq 1,\eta _{k}=1$ in $D_{k},\eta _{k}=0$
outside of $\tilde{D}_{k}$ and for all multiindices $\gamma ,$ $|\partial
^{\gamma }\eta _{k}|\leq C(d,\delta ,|\gamma |)<\infty .$
\end{enumerate}

For $\alpha \in (0,2],\lambda \geq 0,k\geq 1,$ denote%
\begin{eqnarray*}
\mathcal{A}^{(\alpha ),k}f(x) &=&\mathcal{A}_{x_{k}}^{(\alpha )}f(x),%
\mathcal{L}_{\lambda }^{(\alpha ),k}f(x)=(\frac{\partial }{\partial t}+%
\mathcal{A}^{(\alpha ),k}-\lambda )f(x), \\
E^{(\alpha ),k}f(x) &=&\int [f(x+y)-f(x)][\eta _{k}(x+y)-\eta _{k}(x)]\tilde{%
m}_{\alpha }(x_{k},y)\frac{dy}{|y|^{d+\alpha }}, \\
E_{k,1}^{(\alpha )}f(x) &=&\int [f(x+y)-f(x)][\eta _{k}(x+y)-\eta _{k}(x)]%
\frac{dy}{|y|^{d+\alpha }}, \\
F^{(\alpha ),k}f(x) &=&f(x)\mathcal{A}^{(\alpha ),k}\eta
_{k}(x),F_{1}^{(\alpha ),k}f(x)=f(x)|\partial |^{\alpha }\eta _{k}(x),x\in 
\mathbf{R}^{d}.
\end{eqnarray*}

We will need to estimate these operators.

\begin{lemma}
\label{l8}Let $\alpha \in (0,2]$ and $\beta >0, \beta \notin \mathbf{N}_{0}.$
Then

\begin{itemize}
\item[\textup{a)}] for each $\varepsilon >0$, there exists a constant $%
C_{\varepsilon }$ such that for all $f\in C^{\beta }(\mathbf{R}^{d})$, 
\begin{equation*}
\sup_{k}\big( |E_{k}^{(\alpha )}f|_{\beta }+|E_{k,1}^{(\alpha )}f|_{\beta }%
\big) \leq \varepsilon ||\partial |^{\alpha }f|_{\beta }+C_{\varepsilon
}|f|_{\beta -[\beta ]};
\end{equation*}

\item[\textup{b)}] There is a constant $N=N(\alpha ,\beta ,d,\delta
,M^{(\alpha )})$ such that for all $f\in C^{\beta }(\mathbf{R}^{d})$, 
\begin{equation*}
\sup_{k}\big( |F_{k}^{(\alpha )}f|_{\beta }+|F_{k,1}^{(\alpha )}f|_{\beta }%
\big) \leq N|f|_{\beta }.
\end{equation*}
\end{itemize}
\end{lemma}

\begin{proof}
For any $\kappa >0$%
\begin{eqnarray*}
E^{(\alpha ),k}f(x) &=&\int_{0}^{1}\int_{0}^{1}\int_{|y|\leq \kappa }(\nabla
f(x+sy),y)(\nabla \eta _{k}(x+ry),y)\mu _{k}^{(\alpha )}(dy)drds\mathbf{1}%
_{\{\alpha \geq 1\}} \\
&&+\int_{0}^{1}\int_{|y|\leq \kappa }[f(x+y)-f(x)](\nabla \eta
_{k}(x+ry),y)\mu _{k}^{(\alpha )}(dy)dr\mathbf{1}_{\{\alpha <1\}} \\
&&+\int_{|y|>\kappa }[f(x+y)-f(x)][\eta _{k}(x+y)-\eta _{k}(x)]\mu
_{k}^{(\alpha )}(dy),
\end{eqnarray*}%
where $\mu _{k}^{(\alpha )}(dy)=\tilde{m}_{\alpha }(x_{k},y)|y|^{-d-\alpha
}dy$. Clearly, 
\begin{eqnarray*}
|E^{(\alpha ),k}f|_{\beta } &\leq &C\Big\{\mathbf{1}_{\{\alpha \geq
1\}}|\nabla f|_{\beta ;p}\int_{|y|\leq \kappa }|y|^{-d-\alpha +2}dy \\
&&+|f|_{\beta }\Big(\mathbf{1}_{\{\alpha <1\}}\int_{|y|\leq \kappa
}|y|^{-d-\alpha +1}dy+\int_{|y|>\kappa }|y|^{-d-\alpha }dy\Big)\Big\}
\end{eqnarray*}%
The part b) is straightforward.
\end{proof}

\subsubsection{Proof of Theorem \protect\ref{thm:StoCP}}

It can be easily seen that for any $f\in C^{\alpha +\beta }(\mathbf{R}^{d})$%
, 
\begin{eqnarray}
\sup_{x}|f(x)| &\leq &\sup_{x}\sup_{k}\eta
_{k}(x)|f(x)|=\sup_{k}\sup_{x}|\eta _{k}(x)f(x)|,  \notag \\
|f|_{\beta } &\leq &\sup_{k}|\eta _{k}f|_{\beta
}+N\sup_{x}|f(x)|,\sup_{k}|\eta _{k}f|_{\beta }\leq |f|_{\beta
}+N\sup_{x}|f(x)|  \label{f14_150}
\end{eqnarray}

Indeed, for each $x,y\in \mathbf{R}^{d},$%
\begin{eqnarray*}
&&|\partial ^{\lbrack \beta ]}f(x)-\partial ^{\lbrack \beta ]}f(y)| \\
&=&\sup_{k}\eta _{k}(x)|\partial ^{\lbrack \beta ]}f(x)-\partial ^{\lbrack
\beta ]}f(y)|=\sup_{k}|\eta _{k}(x)\partial ^{\lbrack \beta ]}f(x)-\eta
_{k}(x)\partial ^{\lbrack \beta ]}f(y)| \\
&\leq &\sup_{k}|\partial ^{\lbrack \beta ]}\eta _{k}(x)f(x)-\partial
^{\lbrack \beta ]}\eta _{k}(y)f(y)|+\sup_{k}|(\eta _{k}(y)-\eta _{k}(x))u(y)|
\\
&\leq &\sup_{k}|\partial ^{\lbrack \beta ]}\big(\eta _{k}(x)f(x)\big)%
-\partial ^{\lbrack \beta ]}\big(\eta _{k}(y)f(y)\big)| \\
&&+C|f|_{\beta -1}|x-y|^{\beta -[\beta ]}.
\end{eqnarray*}%
The second inequality in (\ref{f14_150}) then follows. Similarly we can
prove the last inequality in (\ref{f14_150}).

By (\ref{f14_150}) and Lemma 11 in \cite{MiZ10}, 
\begin{eqnarray*}
|f|_{\alpha +\beta } &\leq &C\sup_{x}|f(x)|+||\partial |^{\alpha }f|_{\beta
}\leq \sup_{k}|\eta _{k}|\partial |^{\alpha }f|_{\beta }+N\sup_{x}|f(x)| \\
&\leq &C[\sup_{k}||\partial |^{\alpha }(\eta _{k}f)|_{\beta
}+\sup_{k}|f|\partial |^{\alpha }\eta _{k}+E_{1}^{(\alpha ),k}f|_{\beta }.
\end{eqnarray*}%
and by Lemma \ref{l8},%
\begin{equation}
|f|_{\alpha +\beta }\leq C\sup_{k}|\eta _{k}u|_{\alpha +\beta }.  \label{in1}
\end{equation}

Let $u\in C^{\alpha +\beta }(H)$ be a solution to (\ref{eqn:cauchy_prf}).
Then $\eta _{k}u$ satisfies the equation%
\begin{eqnarray}  \label{eqk}
\partial _{t}(\eta _{k}u) &=&\mathcal{A}^{(\alpha ),k}(\eta _{k}u)-\lambda
(\eta _{k}u)+\eta _{k}[\mathcal{A}^{(\alpha )}u-\mathcal{A}^{(\alpha ),k}u] 
\notag \\
& & +\eta _{k}\mathcal{B}^{(\alpha )}u+\eta _{k}f +F^{(\alpha
),k}u+E^{(\alpha ),k}u,
\end{eqnarray}%
and by Proposition \ref{prop1},%
\begin{eqnarray*}
|\eta _{k}u|_{\alpha +\beta } \leq C[|\eta _{k}[\mathcal{A}^{(\alpha )}u-%
\mathcal{A}^{(\alpha ),k}u]|_{\beta }+|\eta _{k}B^{(\alpha )}u|_{\beta }
+|\eta _{k}f|_{\beta } +|F^{(\alpha ),k}u|_{\beta }+|E^{(\alpha
),k}u|_{\beta }].
\end{eqnarray*}%
Therefore,%
\begin{equation}
|u|_{\alpha +\beta }\leq C[\sup_{k}|\eta _{k}f|_{\beta }+I^{(\alpha )}],
\label{form00}
\end{equation}%
where%
\begin{eqnarray*}
I^{(\alpha )} = |\eta _{k}[\mathcal{A}^{(\alpha )}u-\mathcal{A}^{(\alpha
),k}u]|_{\beta }+|\eta _{k}B^{(\alpha )}u|_{\beta } +|F^{(\alpha
),k}u|_{\beta }+|E^{(\alpha ),k}u|_{\beta }.
\end{eqnarray*}%
By Corollary 14 \cite{MiZ10}, 
\begin{equation*}
|\eta _{k}[\mathcal{A}^{(\alpha )}u-\mathcal{A}^{(\alpha ),k}u]|_{\beta
}\leq C\delta ^{\beta }|u|_{\alpha +\beta }.
\end{equation*}%
Using the estimates of Lemma \ref{l8} and Proposition \ref{b1}, we obtain
that for each $\varepsilon >0\,\ $there is a constant $C_{\varepsilon }$
such that%
\begin{equation}
I^{(\alpha )}\leq \varepsilon |u|_{\alpha +\beta }+C_{\varepsilon
}\sup_{t,x}|u|.  \label{nf6}
\end{equation}

By (\ref{form00}), 
\begin{equation}
|u|_{\alpha +\beta }\leq C[|f|_{\beta }+|u|_{\beta }].  \label{form01}
\end{equation}

On the other hand, (\ref{eqk}) holds and by Proposition \ref{prop1},%
\begin{equation*}
|u|_{\beta }\leq \sup_{k}|\eta _{k}u|_{\beta }\leq \mu (\lambda )[|f|_{\beta
}+I_{(\alpha )}],
\end{equation*}%
where $\mu (\lambda )\rightarrow 0$ as $\lambda \rightarrow \infty $. Hence,
by (\ref{nf6}), 
\begin{equation}
|u|_{\beta }\leq C\mu (\lambda )[|f|_{\beta }+|u|_{\alpha +\beta }].
\label{form02}
\end{equation}%
The inequalities (\ref{form01}) and (\ref{form02}) imply that there exist $%
\lambda_{0}$ with $0 < \lambda_0 \le \lambda$ and a constant $C$ independent
of $u$ such that%
\begin{equation}
|u|_{\alpha +\beta }\leq C|f|_{\beta }  \label{form03}
\end{equation}%
If $u\in C^{\alpha +\beta }(H)$ solves equation (\ref{eqn:cauchy_prf}) with $%
\lambda \leq \lambda _{0}$, then $\tilde{u}(t,x)=e^{-(\lambda _{0}-\lambda
)t}u(t,x)$ solves the same equation with $\lambda _{0},$ and by (\ref{form03}%
),%
\begin{equation*}
|u|_{\alpha +\beta }\leq e^{(\lambda _{0}-\lambda )T}|\tilde{u}|_{\alpha
+\beta }\leq Ce^{(\lambda _{0}-\lambda )T}|f|_{\beta }.
\end{equation*}%
Thus (\ref{form03}) holds for all $\lambda \geq 0.$ Again by Proposition \ref%
{prop1} and (\ref{in1}), there is a constant $C$ such that for all $s\leq
t\leq T,$ 
\begin{eqnarray*}
|u(t,\cdot )-u(s,\cdot )|_{\alpha /2+\beta } \leq \sup_{k}|\eta
_{k}u(t,\cdot )-\eta _{k}u(s,\cdot )|_{\alpha /2+\beta } \leq
C(t-s)^{1/2}[|f|_{\beta }+|u|_{\alpha +\beta }].
\end{eqnarray*}%
Therefore there is a constant $C$ such that for all $s\leq t\leq T,$%
\begin{equation*}
|u(t,\cdot )-u(s,\cdot )|_{\alpha /2+\beta }\leq C(t-s)^{1/2}|f|_{\beta }.
\end{equation*}

Let $\,\mathcal{L=A}_{x}^{(\alpha )}+\mathcal{B}_{x}^{(\alpha )}-\lambda
,\tau \in \big[0,1\big],$ and%
\begin{equation*}
\mathcal{L}_{\tau }u=\tau \mathcal{L}u+\big(1-\tau \big)|\partial |^{\alpha
}u.
\end{equation*}%
We introduce the space $\tilde{C}^{\alpha +\beta }\big(H\big)$ of functions $%
u\in C^{\alpha +\beta }(H)$ such that for each $\big(t,x\big)$, 
\begin{equation*}
u\big(t,x\big)=\int_{0}^{t}F\big(s,x\big)\,ds,
\end{equation*}%
where $F\in C^{\beta }\big(H\big).$ It is a Banach space with respect to the
norm 
\begin{equation*}
\big\vert u\big\vert_{\alpha +\beta }^{\symbol{126}}=\big\vert u\big\vert%
_{\alpha +\beta }+\big\vert F\big\vert_{\beta }.
\end{equation*}%
Consider the mappings $T_{\tau }:\tilde{C}^{\alpha +\beta }\big(H\big)%
\rightarrow C^{\beta }$ defined by 
\begin{equation*}
u\big(t,x\big)=\int_{0}^{t}F\big(s,x\big)\,ds\longmapsto F-\mathcal{L}_{\tau
}u.
\end{equation*}%
Obviously, for some constant $C$ independent of $\tau $%
\begin{equation*}
\big\vert T_{\tau }u\big\vert_{\beta }\leq C\big\vert u\big\vert_{\alpha
+\beta }^{\symbol{126}}.
\end{equation*}%
On the other hand, there is a constant $C$ independent of $\tau $ such that
for all $u\in \tilde{C}^{\alpha +\beta }\big(H\big)$%
\begin{equation}
\big\vert u\big\vert_{\alpha +\beta }^{\symbol{126}}\leq C\big\vert T_{\tau
}u\big\vert_{\beta }.  \label{cp1}
\end{equation}%
Indeed, 
\begin{equation*}
u\big(t,x\big)=\int_{0}^{t}F\big(s,x\big)\,ds=\int_{0}^{t}\big(L_{\tau }u+(F-%
\mathcal{L}_{\tau }u)\big)\,ds.
\end{equation*}%
According to (\ref{form03}), there is a constant $C$ independent of $\tau $
such that 
\begin{equation}
\big\vert u\big\vert_{\alpha +\beta }\leq C\big\vert T_{\tau }u\big\vert%
_{\beta }=C\big\vert F-\mathcal{L}_{\tau }u\big\vert_{\beta }.  \label{cp2}
\end{equation}%
Thus, 
\begin{eqnarray*}
|u|_{\alpha +\beta }^{\symbol{126}} &=&|u|_{\alpha +\beta }+|F|_{\beta }\leq
|u|_{\alpha +\beta }+|F-\mathcal{L}_{\tau }u|_{\beta }+|\mathcal{L}_{\tau
}u|_{\beta } \\
&\leq &C|u|_{\alpha +\beta }+|F-\mathcal{L}_{\tau }u|_{\beta }\leq C|F-%
\mathcal{L}_{\tau }u|_{\beta }=C|T_{\tau }u|_{\beta },
\end{eqnarray*}%
and (\ref{cp1}) follows. Since $T_{0}$ is an onto map, by Theorem 5.2 in 
\cite{GiT83}, all the $T_{\tau }$ are onto maps and the statement follows.

\subsubsection{Proof of Corollary \protect\ref{lcornew1}}

By Corollary 14 in \cite{MiZ10} and Proposition \ref{b1}, for $g\in
C^{\alpha +\beta }(\mathbf{R}^{d})$, $|\mathcal{A}^{(\alpha )}g|_{\beta
}\leq C|g|_{\alpha +\beta }$ and $|\mathcal{B}^{(\alpha )}g|_{\beta }\leq
C|g|_{\alpha +\beta }$ with a constant $C$ independent of $f$ and $g$. It
then follows from (\ref{eqn:cauchy_prf}) that there exists a unique solution 
$\tilde{v}\in C^{\alpha +\beta }(H)$ to the Cauchy problem 
\begin{eqnarray}
\big( \partial _{t}+\mathcal{A}_{x}^{(\alpha )}+\mathcal{B}_{x}^{(\alpha )}%
\big) \tilde{v}(t,x) &=&f(t,x)-\mathcal{A}_{x}^{(\alpha )}g(x)-\mathcal{B}%
_{x}^{(\alpha )}g(x),  \notag \\
\tilde{v}(T,x) &=&0  \label{eqn:cauchy_v}
\end{eqnarray}%
and $|\tilde{v}|_{\alpha +\beta }\leq C\big( |g|_{\alpha +\beta }+|f|_{\beta
}\big) $ with $C$ independent of $f$ and $g$. Let $v(t,x)=\tilde{v}%
(t,x)+g(x) $, where $\tilde{v}$ is the solution to problem (\ref%
{eqn:cauchy_v}). Then $v $ is the unique solution to the Cauchy problem (\ref%
{maf8}) and $|v|_{\alpha +\beta }\leq C(|g|_{\alpha +\beta }+|f|_{\beta })$.

\begin{remark}
\label{rlast}If the assumptions of Corollary \ref{lcornew1} hold and $v\in
C^{\alpha +\beta }(H)$ is the solution to $(\ref{maf8})$, then $\partial
_{t}v=f-\mathcal{A}_{x}^{(\alpha )}v-\mathcal{B}_{x}^{(\alpha )}v$, and
according to Corollary 14 in \textup{\cite{MiZ10}} and Proposition \ref{b1}, 
$|\partial _{t}v|_{\beta }\leq C(|g|_{\alpha +\beta }+|f|_{\beta }).$
\end{remark}


\section{One Step Estimate and Proof of the Main Result}


The following Lemma provides a one-step estimate of the conditional
expectation of an increment of the Euler approximation.

\begin{lemma}
\label{lem:expect}Let $\alpha \in (0,2]$, $\beta >0,\beta \notin \mathbf{N}$%
, and $\delta >0$. Assume \textup{A1-A4$(\beta )$} hold. Then there exists a
constant $C$ such that for all $f\in C^{\beta }(\mathbf{R}^{d}),$%
\begin{equation*}
\big\vert\mathbf{E}\big[f(Y_{s})-f(Y_{\tau _{i_{s}}})|\mathcal{F}_{\tau
_{i_{s}}}\big]\big\vert\leq C|f|_{\beta }\delta ^{\kappa (\alpha ,\beta
)},\forall s\in \lbrack 0,T],
\end{equation*}%
where $i_{s}=i$ if $\tau _{i}\leq s<\tau _{i+1}$ and $\kappa (\alpha ,\beta
) $ is as defined in Theorem \ref{thm:main}.
\end{lemma}

The proof of Lemma~\ref{lem:expect} is based on applying It\^{o}'s formula
to $f(Y_{s})-f(Y_{\tau _{i_{s}}})$, $f\in C^{\beta }(\mathbf{R}^{d})$. If $%
\beta >\alpha $, by Remark \ref{lrenew1} and It\^{o}'s formula, the
inequality holds. If $\beta <\alpha $, we first smooth $f$ by using $w\in
C_{0}^{\infty }(\mathbf{R}^{d}),$ a nonnegative smooth function with support
on $\{|x|\leq 1\}$ such that $w(x)=w(|x|)$, $x\in \mathbf{R}^{d},$ and $\int
w(x)dx=1$ (see (8.1) in \cite{Fol99})$.$ Note that, because of the symmetry, 
\begin{equation}
\int_{\mathbf{R}^{d}}x^{i}w(x)dx=0,i=1,\ldots ,d.  \label{ff26}
\end{equation}

For $x\in \mathbf{R}^{d}$ and $\varepsilon \in (0,1)$, define $%
w^{\varepsilon }(x)=\varepsilon ^{-d}w\big( \frac{x}{\varepsilon }\big) $
and the convolution 
\begin{equation}
f^{\varepsilon }(x)=\int f(y)w^{\varepsilon }(x-y)dy=\int
f(x-y)w^{\varepsilon }(y)dy,x\in \mathbf{R}^{d}.  \label{maf7}
\end{equation}


\subsection{Some Auxiliary Estimates}

In \cite{MiZ10} the following estimates for $\mathcal{A}_{z}^{(\alpha )}$
and $f^{\varepsilon }$ were proved.

\begin{lemma}
\label{lnew2} $($Lemma 21 in \textup{\cite{MiZ10}}$)$ Let $\alpha \in (0,2)$%
, $\beta <\alpha$, $\beta \neq 1$, and $\varepsilon \in (0,1)$. Then

\begin{enumerate}
\item[\textup{(i)}] there exists a constant $C$ such that for all $f\in
C^{\beta }(\mathbf{R}^{d}),x\in \mathbf{R}^{d}$, 
\begin{equation*}
|f^{\varepsilon }(x)-f(x)|\leq C\varepsilon ^{\beta }|f|_{\beta };
\end{equation*}

\item[\textup{(ii)}] there exists a constant $C$ such that for all $z,x\in 
\mathbf{R}^{d},$ 
\begin{equation}
|\mathcal{A}_{z}^{(\alpha )}f^{\varepsilon }(x)|\leq C\varepsilon ^{-\alpha
+\beta }|f|_{\beta }  \label{ff29}
\end{equation}%
and in particular, for all $f\in C^{\beta }(\mathbf{R}^{d}),z,x\in \mathbf{R}%
^{d}$, 
\begin{equation}
|\partial ^{\alpha }f^{\varepsilon }(x)|\leq C\varepsilon ^{-\alpha +\beta
}|f|_{\beta };  \label{maf5}
\end{equation}

\item[\textup{(iii)}] for $k,l=1,\ldots ,d,x\in \mathbf{R}^{d},$%
\begin{eqnarray}
|\partial _{k}f^{\varepsilon }(x)| &\leq &C\varepsilon ^{-1+\beta
}|f|_{\beta },\mbox{ if }\beta <1,  \label{ff30} \\
|f^{\varepsilon }|_{1} &\leq &C|f|_{1},  \notag \\
|\partial _{kl}^{2}f^{\varepsilon }(x)| &\leq &C\varepsilon ^{-2+\beta
}|f|_{\beta },\mbox{ if }\beta <2,\text{ }  \notag
\end{eqnarray}%
and%
\begin{eqnarray}
|f^{\varepsilon }|_{\alpha } &\leq &C\varepsilon ^{-\alpha +\beta
}|f|_{\beta },\mbox{ if } \alpha \in (1,2), \beta \in (0,1],  \label{maf5'}
\\
|\partial ^{\alpha -1}\nabla f^{\varepsilon }(x)| &\leq &C\varepsilon
^{-\alpha +\beta }|f|_{\beta },\mbox{ if } \alpha \in (1,2), \beta \in
(1,\alpha ).  \label{maf6}
\end{eqnarray}
\end{enumerate}
\end{lemma}

\begin{corollary}
\label{coro3}Assume $a(x)$ and 
\begin{equation*}
\int [\mathbf{1}_{U_{1}}(\upsilon )|l_{\alpha }(x,\upsilon )|^{\alpha }+%
\mathbf{1}_{U_{1}^{c}}(\upsilon )|l_{\alpha }(x,\upsilon )|^{\alpha \wedge
1}\wedge 1]\pi (d\upsilon ),
\end{equation*}%
are bounded, $\varepsilon \in (0,1)$. Then there exists a constant $C$ such
that for all $z,x\in \mathbf{R}^{d}$, $f\in C^{\beta }(\mathbf{R}^{d}),$%
\begin{equation*}
|\mathcal{B}_{z}^{(\alpha )}f^{\varepsilon }(x)|\leq C\varepsilon ^{-\alpha
+\beta }|f|_{\beta }.
\end{equation*}
\end{corollary}

\begin{proof}
If $\beta <\alpha <1$, by Lemma \ref{r1},%
\begin{equation*}
f^{\varepsilon }(x+y)-f^{\varepsilon }(x)=\int k^{(\alpha )}(y,y^{\prime
})\partial ^{\alpha }f^{\varepsilon }(x-y^{\prime })dy^{\prime },
\end{equation*}%
and by Lemma \ref{lnew2}, (\ref{maf5}), 
\begin{equation}
|f^{\varepsilon }(x+y)-f^{\varepsilon }(x)|\leq C\varepsilon ^{-\alpha
+\beta }|f|_{\beta }(|y|^{\alpha }\wedge 1),x,y\in \mathbf{R}^{d}
\label{ff32}
\end{equation}%
and%
\begin{eqnarray*}
|f^{\varepsilon }(x+l_{\alpha }(x,\upsilon ))-f^{\varepsilon }(x)| &\leq
&C\varepsilon ^{-\alpha +\beta }|f|_{\beta }(|l_{\alpha }(x,\upsilon
)|^{\alpha }\wedge 1) \\
&\leq &C\varepsilon ^{-\alpha +\beta }|f|_{\beta }[\mathbf{1}%
_{U_{1}}(\upsilon )|l_{\alpha }(x,\upsilon )|^{\alpha } +\mathbf{1}%
_{U_{1}^{c}}(\upsilon )(|l_{\alpha }(x,\upsilon )|^{\alpha }\wedge 1)].
\end{eqnarray*}

If $\beta <\alpha =1$, by Lemma \ref{lnew2}(ii) and (\ref{ff30}),%
\begin{eqnarray}
|f^{\varepsilon }(x+y)-f^{\varepsilon }(x)| &\leq &C\sup_{x}[f(x)|+|\nabla
f^{\varepsilon }(x)|](|y|\wedge 1)  \label{ff33} \\
&\leq &C\varepsilon ^{-1+\beta }|f|_{\beta }(|y|\wedge 1),x,y\in \mathbf{R}%
^{d}  \notag
\end{eqnarray}%
and%
\begin{eqnarray*}
|f^{\varepsilon }(x+l_{1}(x,\upsilon ))-f^{\varepsilon }(x)| &\leq
&C\varepsilon ^{-1+\beta }|f|_{\beta }(|l_{1}(x,\upsilon )|\wedge 1) \\
&\leq &C\varepsilon ^{-1+\beta }|f|_{\beta }[\mathbf{1}_{U_{1}}(\upsilon
)|l_{1}(x,\upsilon )| +\mathbf{1}_{U_{1}^{c}}(\upsilon )(|l_{1}(x,\upsilon
)|\wedge 1)].
\end{eqnarray*}

Assume $\alpha \in (1,2),$ then for $x,y\in \mathbf{R}^{d},$%
\begin{equation}
f^{\varepsilon }(x+y)-f^{\varepsilon }(x)-(\nabla f^{\varepsilon
}(x),y)=\int_{0}^{1}\big(\nabla f^{\varepsilon }(x+sy)-\nabla f^{\varepsilon
}(x),y\big)ds.  \label{ff34}
\end{equation}%
If $\beta \in (1,\alpha )$, then by Lemmas \ref{r1}, \ref{lnew2} and (\ref%
{maf5'}), for $x,y^{\prime }\in \mathbf{R}^{d},$%
\begin{eqnarray}
|\nabla f^{\varepsilon }(x+y^{\prime })-\nabla f^{\varepsilon }(x)| &\leq
&C\sup_{x}|\partial ^{\alpha -1}\nabla f^{\varepsilon }(x)|~|y^{\prime
}|^{\alpha -1}  \label{ff35} \\
&\leq &C\varepsilon ^{-\alpha +\beta }|f|_{\beta }|y^{\prime }|^{\alpha -1}.
\notag
\end{eqnarray}%
If $\beta >\alpha >1$, then directly%
\begin{equation}
|\nabla f^{\varepsilon }(x+y^{\prime })-\nabla f^{\varepsilon }(x)|\leq
C|f|_{\beta }|y^{\prime }|^{\alpha -1}.  \notag
\end{equation}

If $\beta \in (0,1],\alpha \in (1,2)$, then by Lemma \ref{lnew2}, (\ref{maf6}%
),%
\begin{equation}
|\nabla f^{\varepsilon }(x+y^{\prime })-\nabla f^{\varepsilon }(x)|\leq
C\varepsilon ^{-\alpha +\beta }|y^{\prime }|^{\alpha -1}|f|_{\beta }
\label{ff36}
\end{equation}%
Applying (\ref{ff35}), (\ref{ff36}) to (\ref{ff34}) we have for $x,y\in 
\mathbf{R}^{d},$ 
\begin{equation*}
|f^{\varepsilon }(x+y)-f^{\varepsilon }(x)-(\nabla f^{\varepsilon
}(x),y)|\leq C\varepsilon ^{-\alpha +\beta }|y|^{\alpha }|f|_{\beta }.
\end{equation*}%
Hence,%
\begin{eqnarray*}
\mathbf{1}_{U_{1}}(\upsilon )|f^{\varepsilon }(x+l_{\alpha }(x,\upsilon
))-f^{\varepsilon }(x)-(\nabla f^{\varepsilon }(x),l_{\alpha }(x,\upsilon )|
\leq C\varepsilon ^{-\alpha +\beta }|l_{\alpha }(x,\upsilon )|^{\alpha
}|f|_{\beta }.
\end{eqnarray*}%
Also, for $\alpha >1,\beta \in (1,\alpha ),$%
\begin{eqnarray*}
\mathbf{1}_{U_{1}^{c}}(\upsilon )|f^{\varepsilon }(x+l_{\alpha }(x,\upsilon
))-f^{\varepsilon }(x)| &\leq &C|f|_{\beta }\big( |l_{\alpha }(x,\upsilon
)|\wedge 1\big).
\end{eqnarray*}

Therefore, the statement follows by the assumptions and Lemma \ref{lnew2}.
\end{proof}


\subsection{Proof of Lemma \protect\ref{lem:expect}}

If $\beta <\alpha \,,$ define $f^{\varepsilon }$ by (\ref{maf7}) for $%
\varepsilon \in (0,1)$ and apply It\^{o}'s formula (see Remark \ref{lrenew1}%
): for $s\in \lbrack 0,T]$, 
\begin{equation*}
\mathbf{E}[f^{\varepsilon }(Y_{s})-f^{\varepsilon }(Y_{\tau _{i_{s}}})|%
\mathcal{F}_{\tau _{i_{s}}}]=\mathbf{E}\big[\int_{\tau _{i_{s}}}^{s}(%
\mathcal{A}_{Y_{\tau _{i_{s}}}}^{(\alpha )}f^{\varepsilon }(Y_{r})+\mathcal{B%
}_{Y_{\tau _{i_{s}}}}^{(\alpha )}f^{\varepsilon }(Y_{r}))dr|\mathcal{F}%
_{\tau _{i_{s}}}\big].
\end{equation*}%
Hence, by Lemma \ref{lnew2} and Corollary \ref{coro3}, for $\varepsilon \in
(0,1)$,%
\begin{eqnarray*}
|\mathbf{E}[f(Y_{s})-f(Y_{\tau _{i_{s}}})|\mathcal{F}_{\tau _{i_{s}}}]|
&\leq &|\mathbf{E}[(f-f^{\varepsilon })(Y_{s})-(f-f^{\varepsilon })(Y_{\tau
_{i_{s}}})|\mathcal{F}_{\tau _{i_{s}}}]| \\
& & +|\mathbf{E}[f^{\varepsilon }(Y_{s})-f^{\varepsilon }(Y_{\tau _{i_{s}}})|%
\mathcal{F}_{\tau _{i_{s}}}]| \\
&\leq &C(\varepsilon ^{\beta }+\delta \varepsilon ^{-\alpha +\beta
})|f|_{\beta },
\end{eqnarray*}%
with a constant $C$ independent of $\varepsilon ,f$. Minimizing $\varepsilon
^{\beta }+\delta \varepsilon ^{-\alpha +\beta }$ in $\varepsilon \in (0,1)$,
we obtain%
\begin{equation*}
|\mathbf{E}[f(Y_{s})-f(Y_{\tau _{i_{s}}})|\mathcal{F}_{\tau _{i_{s}}}]|\leq
C\delta ^{\kappa (\alpha ,\beta )}|f|_{\beta }.
\end{equation*}

If $\beta >\alpha ,$ we apply It\^{o}'s formula directly (see Remark \ref%
{lrenew1}):%
\begin{equation*}
\mathbf{E}[f(Y_{s})-f(Y_{\tau _{i_{s}}})|\mathcal{F}_{\tau _{i_{s}}}]=%
\mathbf{E}\Big[\int_{\tau _{i_{s}}}^{s}\big(\mathcal{A}_{Y_{\tau
_{i_{s}}}}^{(\alpha )}f(Y_{r})+\mathcal{B}_{Y_{\tau _{i_{s}}}}^{(\alpha
)}f(Y_{r})\big)dr|\mathcal{F}_{\tau _{i_{s}}}\Big].
\end{equation*}%
Hence, by Corollary 14 in \cite{MiZ10} and Lemma \ref{lnew2}, 
\begin{equation*}
|\mathbf{E}[f(Y_{s})-f(Y_{\tau _{i_{s}}})|\mathcal{F}_{\tau _{i_{s}}}]|\leq
C\delta |f|_{\beta }.
\end{equation*}%
The statement of Lemma \ref{lem:expect} follows.


\subsection{Proof of Theorem \protect\ref{thm:main}}

Let $v\in C^{\alpha +\beta }(H)$ be the unique solution to (\ref{maf8}) (see
Corollary \ref{lcornew1}). By It\^{o}'s formula (see Remark \ref{lrenew1}, (%
\ref{ff25})) and (\ref{maf8})),%
\begin{eqnarray*}
\mathbf{E}v(0,X_{0}) &=&\mathbf{E}v(T,X_{T})-\mathbf{E}\int_{0}^{T}(\partial
_{t}v(s,X_{s})+\mathcal{A}_{X_{s}}^{(\alpha )}v(s,X_{s})+\mathcal{B}%
_{X_{s}}^{(\alpha )}v(s,X_{s})]ds \\
&=&\mathbf{E}\big[g(X_{T})-\int_{0}^{T}f(X_{s})ds\big],
\end{eqnarray*}%
and 
\begin{equation}
\mathbf{E}v(0,X_{0})=\mathbf{E}v(0,Y_{0}).  \label{eqn:expect_terminal}
\end{equation}

By Proposition \ref{b1}, Corollaries 14 in \cite{MiZ10}, \ref{lcornew1} and
Remark \ref{rlast},%
\begin{eqnarray}
|\mathcal{A}_{z}^{(\alpha )}v(s,\cdot )|_{\beta }+|\mathcal{B}_{z}^{(\alpha
)}v(s,\cdot )|_{\beta } &\leq &C|v|_{\alpha +\beta }\leq C|g|_{\alpha +\beta
},  \label{maf9} \\
|\partial _{t}v(s,\cdot )|_{\beta } &\leq &C|g|_{\alpha +\beta },s\in
\lbrack 0,T].  \notag
\end{eqnarray}

Then, by It\^{o}'s formula (Remark \ref{lrenew1}, (\ref{ff25})) and
Corollary \ref{lcornew1}, with (\ref{eqn:expect_terminal}) and (\ref{maf9}),
it follows that 
\begin{eqnarray*}
&&\mathbf{E}g(Y_{T})-\mathbf{E}g(X_{T})-\mathbf{E}\int_{0}^{T}f(Y_{\tau
_{i_{s}}})ds+\mathbf{E}\int_{0}^{T}f(X_{s})ds \\
&=&\mathbf{E}v(T,Y_{T})-\mathbf{E}v(0,Y_{0})-\mathbf{E}\int_{0}^{T}f(Y_{_{%
\tau _{i_{s}}}})ds+\mathbf{E}\int_{0}^{T}f(X_{s})ds \\
&=&\mathbf{E}\int_{0}^{T}\big\{\lbrack \partial _{t}v(s,Y_{s})-\partial
_{t}v(s,Y_{\tau _{i_{s}}})] \\
&&+[\mathcal{A}_{Y_{\tau _{i_{s}}}}^{(\alpha )}v(s,Y_{s})-\mathcal{A}%
_{Y_{\tau _{i_{s}}}}^{(\alpha )}v(s,Y_{\tau _{i_{s}}})] \\
&&+[\mathcal{B}_{Y_{\tau _{i_{s}}}}^{(\alpha )}v(s,Y_{s})-\mathcal{B}%
_{Y_{\tau _{i_{s}}}}^{(\alpha )}v(s,Y_{\tau _{i_{s}}})]\big\}ds.
\end{eqnarray*}

Hence, by (\ref{maf9}) and Lemma~\ref{lem:expect}, there exists a constant $%
C $ independent of $g$ such that 
\begin{equation*}
|\mathbf{E}g(Y_{T})-\mathbf{E}g(X_{T})|\leq C\delta ^{\kappa (\alpha ,\beta
)}|g|_{\alpha +\beta }.
\end{equation*}%
The statement of Theorem \ref{thm:main} follows.



\begin{thebibliography}{99}
\bibitem{AbK09} Abels, H. and Kassman, M., The Cauchy problem and the
martingale problem for integro-differential operators with non-smooth
kernels, Osaka J. Math. 46 (2009) 661-683.

\bibitem{BeL76} Bergh, J. and L\"{o}fstr\"{o}m, J., \textit{Interpolation
Spaces. An Introduction}, Springer Verlag, 1976.

\bibitem{Fol99} Folland, G. B., \textit{Real Analysis}, John Wiley, New
York, 1999.

\bibitem{GiT83} Gilbarg, D. and Trudinger, N. S. \emph{Elliptic Partial
Differential Equations of Second Order}. Springer, New York, 1983.

\bibitem{Jac79} Jacod, J., \textit{Calcul Stochastique et Probl\`{e}mes de
Martingales}, \textit{Lecture Notes in Mathematics}, 714, Springer Verlag,
Berlin New York, 1979.

\bibitem{KlP00} Kloeden, P. E. and Platen, E., \textit{Numerical Solution of
Stochastic Differential Equations}, Springer Verlag, 2000.

\bibitem{Kom84} Komatsu, T., On the Martingale Problem for Generators of
Stable Processes with Perturbations, \textit{Osaka J. of Math.} 22-1 (1984)
113-132.

\bibitem{Kub90} Kubilius, K., On the Rate of Convergence of the
Distributions of Semimartingales to the Distribution of Stable Process. In: 
\textit{Probability Theory and Mathematical Statistics} 2 (1990) 22-34.

\bibitem{KuP01} Kubilius, K. and Platen, E., Rate of Weak Convergence of the
Euler Approximation for Diffusion Processes with Jumps, Quantitative Finance
Research Centre, University of Technology. Sydney, \textit{Research Paper
Series} 54, 2001.

\bibitem{MiP88} Mikulevi\v{c}ius, R. and Platen, E., Time Discrete Taylor
Approximations for It\v{o} Processes with Jump Component, \textit{%
Mathematische Nachrichten} 138 (1988.) 93-104.

\bibitem{MiP911} Mikulevi\v{c}ius, R. and Platen, E., Rate of Convergence of
the Euler Approximation for Diffusion Processes, \textit{Mathematische
Nachrichten} 151 (1991) 233-239.

\bibitem{MiP922} Mikulevi\v{c}ius, R. and Pragarauskas, H., On the Cauchy
Problem for Certain Integro-Differential Operators in Sobolev and H\"{o}lder
Spaces, \textit{Lithuanian Mathematical Journal} 32-2 (1992) 238-264.

\bibitem{MiP923} Mikulevi\v{c}ius, R. and Pragarauskas, H., On the
Martingale Problem Associated with Nondegenerate L\'{e}vy Operators, \textit{%
Lithuanian Mathematical Journal} 32-3 (1992) 297-311.

\bibitem{MiP09} Mikulevi\v{c}ius, R. and Pragarauskas, H., On H\"{o}lder
Solutions of the Integro-Differential Zakai Equation, \textit{Stochastic
Processes and their Applications}, 119 (2009) 3319-3355.

\bibitem{MiZ10} Mikulevi\v{c}ius, R. and Zhang, C., On the rate of
convergence of weak Euler approximation for nondegenerate diffusion and jump
processes, \textit{arXiv: 1007.2914v1,} [math.PR]\textit{, }2010, 1-38%
\textit{.}

\bibitem{Mil79} Milstein, G. N., A Method of Second-Order Accuracy
Integration of Stochastic Differential Equations, \textit{Theory of
Probability and its Applications} 23 (1979) 396-401.

\bibitem{Mil86} Milstein, G. N., Weak Approximation of Solutions of Systems
of Stochastic Differential Equations, \textit{Theory of Probability and its
Applications} 30 (1986) 750-766.

\bibitem{Pla99} Platen, E., An Introduction to Numerical Methods for
Stochastic Differential Equations, \textit{Acta Numerica} 8 (1999) 197-246.

\bibitem{plalib} Platen, E. and Bruti-Liberati, N., Numerical Solutions of
Stochastic Differential Equations with Jumps in Finance, Springer Verlag,
2010.

\bibitem{PrT97} Protter, P. E. and Talay, D., The Euler Scheme for L\'{e}vy
Driven Stochastic Differential Equations, \textit{The Annals of Probability}
25 (1997) 393-423.

\bibitem{Tal84} Talay, D., Efficient Numerical Schemes for the Approximation
of Expectations of Functionals of the Solution of a S.D.E. and Applications,
In:\ \textit{Filtering and Control of Random Processes}, \textit{Lecture
Notes in Control and Information Sciences} 61 (1984) 294-313.

\bibitem{Tal86} Talay, D., Discretization of a Stochastic Differential
Equation and Rough Estimate of the Expectations of Functionals of the
Solution, \textit{ESAIM: Mathematical Modelling and Numerical Analysis - Mod%
\'{e}lisation Math\'{e}matique et Analyse Num\'{e}rique} 20 (1986) 141-179.

\bibitem{Tri83} Triebel, H., \textit{Theory of Function Spaces}. Birkhaueser
Verlag, 1983.

\bibitem{Tri92} Triebel, H., \textit{Theory of Function Spaces II}.
Birkhaueser Verlag, 1992.
\end{thebibliography}
\end{document}